\documentclass[12pt]{amsart}
\usepackage{amssymb,graphicx,amscd,accents,array,url}
\usepackage[usenames,dvipsnames]{color}

   \usepackage{mathrsfs}

\calclayout
\usepackage{color}

\usepackage[english]{babel}
\numberwithin{equation}{section}
\allowdisplaybreaks

\newcommand{\VV}{ \mathcal{V}}
\newcommand{\NN}{ \mathcal{N}}
\newcommand{\FF}{ \mathcal{F}}
\newcommand{\SN}{ \mathcal{S}}

\newcommand{\NSS}{\NN_{\SS}}
\newcommand{\RR}{ \Pi}
\newcommand{\DD}{ \mathcal{D}}
\newcommand{\MM}{ \mathcal{M}}

\newcommand\E{{E}}

\newcommand\bgrad{\textrm{\bf grad}}
\renewcommand\div{{\rm div}}
\newcommand\bcurl{\textrm{\bf curl}}
\newcommand\curl{\textrm{\bf curl}}
\newcommand\brot{\textrm{\bf rot}}
\newcommand\rot{\textrm{\rm rot}}

\newcommand{\dE}{\,{\rm d}\E}

\newcommand{\ds}{\,{\rm d}s}

\newcommand{\de}{\,{\rm d}e}
\newcommand{\df}{\,{\rm d}f}

\newcommand\nn{\boldsymbol n}
\renewcommand\tt{\boldsymbol t}

\newcommand\xx{\boldsymbol x}

\newcommand\pp{\boldsymbol p}
\newcommand\qq{\boldsymbol q}

\newcommand\vphi{\varphi}
\newcommand\vv{\boldsymbol v}
\newcommand\vtauf{{\boldsymbol v}^{\tau_f}}

\newcommand\cc{\boldsymbol c}
\newcommand\uu{\boldsymbol u}
\newcommand\zz{\boldsymbol z}

\newcommand\ww{\boldsymbol w}
\newcommand\bc{\boldsymbol c}

\newcommand\Y{{\mathscr V}}

\newcommand\bV{\overline{V}}
\renewcommand\bV{{\bf {V}}}

\newcommand{\calB}{ \mathcal{B}}

\newcommand{\calD}{ \mathcal{D}}
\newcommand{\calC}{ \mathcal{C}}

\renewcommand{\O}{ {\mathcal O}}
\newcommand{\Z}{ {\mathcal Z}}

\renewcommand\SS{{\mathscr S}}

\renewcommand\S{\mathscr{S}}
\newcommand\R{\mathbb{R}}
\renewcommand{\P}{ {\mathbb P}}
\newcommand{\Q}{ {\mathbb Q}}

\newtheorem{thm}{Theorem}[section]
\newtheorem{prop}[thm]{Proposition}
\newtheorem{lem}[thm]{Lemma}

\newtheorem{defn}[thm]{Definition}
\newtheorem{remark}{Remark}

\title{Serendipity Face and Edge VEM Spaces}

\author{
L. Beir\~ao da Veiga$^{[1]}$, F. Brezzi$^{[2]}$, L.D. Marini$^{[3]}$, A. Russo$^{[1]}$
}

\address{{\it Louren\c co Beir\~ao da Veiga} - Dipartimento di Matematica e Applicazioni, Universit\`a di Milano--Bicocca,
Via Cozzi 53, I-20153, Milano, Italy \\
and IMATI del CNR, Via Ferrata 1, 27100 Pavia, Italy
lourenco.beirao@unimib.it}

\address{{\it Franco Brezzi} - IMATI del CNR,
Via Ferrata 5,
27100 Pavia, Italy\\
brezzi@imati.cnr.it}

\address{{\it Luisa Donatella Marini} - Dipartimento di Matematica, Universit\`a di Pavia,
and IMATI del CNR, Via Ferrata 1, 27100 Pavia, Italy\\
marini@imati.cnr.it}

\address{{\it Alessandro Russo} - Dipartimento di Matematica e Applicazioni, Universit\`a di Milano--Bicocca,
Via Cozzi 53, I-20153, Milano, Italy\\
and IMATI del CNR, Via Ferrata 1, 27100 Pavia, Italy
alessandro.russo@unimib.it}

\date{\today}

\begin{document}

\begin{abstract}
We  extend  the basic idea of Serendipity Virtual Elements from the previous case (by the same authors)
of nodal ($H^1$-conforming) elements, to a more general framework. Then we apply the general strategy to the case of   $H(\div)$ and $H(\bcurl)$ conforming Virtual Element Methods,
in two and three dimensions.
\end{abstract}

\bigskip
\maketitle



\section{Introduction}


Virtual Element Methods (VEM) were introduced a few years ago (\cite{volley}, \cite{VEM-elasticity}, \cite{Brezzi:Marini:plates}, \cite{projectors}) as a new interpretation of Mimetic Finite Differences (MFD)  (see {\cite{MFD22},} \cite{MFD23} and the references therein) that allowed, in a suitable sense, a generalization of 
classical Finite Element Methods (FEM) to polygonal and polyhedral decompositions. More recently, they underwent rapid developments, with extension to various problems (see  \cite{variable-primal}, \cite{mixed-variables},  \cite{BeiraoLovaMora}, 
  \cite{CH-VEM}, \cite{Benedetto-VEM-2}, \cite{Berrone-VEM},  \cite{Paulino-VEM},  \cite{VemSteklov}).

Contrary to MFD,
(and similarly to FEM) Virtual Elements are a Galerkin method, using the variational formulation of the continuous problems in suitable finite dimensional spaces. Contrary to Finite Elements (and similarly
to MFD) they can be used on very general decompositions, both in 2 and 3 dimensions, and are very
robust with respect to element distortions, hanging nodes, and so on. Similarly to other methods
for polytopes (see e.g. {\cite{Arroyo-Ortiz}}, \cite{Bishop},  \cite{Floater:Acta}, 
\cite{Gillette-14}, \cite{BEM-weisser}, \cite{Sukumar:Malsch:2006}, \cite{Wachspress2016})  they use finite dimensional spaces that, within each element, contain functions that are not polynomials. 

Unlike these previous methods, however, with VEM these functions need not to be computed (not even in a rough way) inside the elements, but some of their properties (averages, polynomial projections,
and the like) are computed {\it exactly} starting from the data, and this allows (at least in problems
with constant coefficients) the construction of schemes that satisfy the Patch Test {\it exactly}.

We also point out the interesting connections of Virtual elements with several other important classes of methods
based on a split approximation of the same variables (at the boundary and in the interior), such as
the many variants of the quite successful Hybridizable Discontinuous Galerkin (HDG) (see e.g. \cite{Cockburn-Jay-Sayas},  \cite{Cockburn-IMU})
or the newer interesting Hybrid High Order methods (see e.g.  \cite{DiPietro-Ern-Lemaire}, \cite{DiPietro-Ern-1}). These connections deserve to be further
investigated, and in particular the question of {\it which approach would be preferable for each class of problems} seems, to us, of paramount important for future studies.

The $H(\div)$ and $H(\bcurl)$ conforming  variants of Virtual Elements were introduced in \cite{BFM} and successively extended in \cite{super-misti}. Their natural Mixed Finite Element counterparts are the classical  Raviart-Thomas ({\it RT}) or Brezzi-Douglas-Marini (\textit{BDM}) elements for the $H(\div)$-conforming case  and  the equally classical N\'ed\'elec elements of first and second kind ({\it N1} and {\it N2}, respectively), for the $H(\bcurl)$-conforming case. Compared to them,  Mixed Virtual Elements exhibit a much better robustness with respect
to the element geometry, but often a bigger number of degrees of freedom (see also, for instance, \cite{Bo-Bre-For} for definitions and properties of above Finite Element spaces, and Figures \ref{clas-tri-f} to
\ref{clas-qua1} here in the next sections for a comparison with VEMs). This justifies the effort to eliminate some internal degrees of freedom and, for $H(\bcurl)$-conforming polyhedrons, also some of the degrees of freedom internal to faces. We are doing this here, following a Serendipity-like strategy, in the stream of what has been done, for instance,  in \cite{A-A}, \cite{SERE-Nod}.

Here we slightly generalise the  $H(\div)$ and $H(\curl)$ conforming spaces presented in \cite{super-misti}, and identify different degrees of freedom, more suited to introduce their Serendipity variants.
We point out that in  \cite{super-misti} we concentrated on spaces and degrees of freedom that allow, on each $d$-dimensional polytope $\E$ (for $d=2$ or $3$),  the computation of the  $L^2$-projection operator on the space $(\P_k(\E))^d$ of vector valued polynomials of degree $\le k$, while here we consider also the possibility of having a so-called $B$-compatible operator (also known as {\it Fortin-type interpolator}), that, as is well known, is crucial in proving the {\it inf-sup}-condition in a number of different circumstances. 

In particular, we consider several types of vector-valued spaces, with different degrees for the boundary, for the divergence, and for the $\bcurl$, that include the polytopal analogues of {\it RT}  and  {\it N1} spaces, as well as the analogues of \textit{BDM} or {\it N2} elements and t{\it RT}he analogues of the Brezzi-Douglas-Fortin-Marini ({\it BDFM}) elements, together with many other possibilities, as it was briefly indicated at the end of \cite{super-misti}.

Here too we limit ourselves to the description of the  {\it local} spaces, on a generic polygon or polyhedron $\E$. The definition of the global spaces (on a decomposition made by several polytopes, respecting  the $H(\div)$ or                                                                                                                                                                                                                                                                                                                                                                                                                                                            the $H(\bcurl)$ conformity) is then immediate. The application of these elements to the approximation of PDE problems in mixed formulation (partly trivial, partly non trivial) will be discussed somewhere else, with error estimates and various additional properties.

Regarding, for several types of  problems, the interest of polytopal decompositions with other numerical approaches, we refer for instance to \cite{Droniou-gradient},  \cite{Fries:Belytschko:2010}, \cite{POLY-TOPOP21}, \cite{GFEM146}, \cite{MFD23}
and to the references therein.

An outline of the paper is as follows. In the next section we will recall a few basic definitions and some properties of polynomial spaces that will be useful in the sequel.  In Section \ref{seregen} we present a rather general framework that we are going to use in order to construct Serendipity-like variants of  local finite dimensional spaces. We note however that the approach goes beyond the particular case of Virtual Element Spaces, and could have an interest  of its own in other situations.   In Section \ref{face2d} we recall first the definition of our $H(\div)$-conforming 2-dimensional elements, which  slightly generalize the previous
\cite{BFM} and \cite{super-misti} cases,  and we show some comparisons with {\it RT}, \textit{BDM}, and {\it BDFM} Finite Elements).
At the end of this Section we also recall the definition of $B$-compatible interpolators, that are very useful for proving {\it inf-sup} conditions. We devote the next Section \ref{sereface} to the construction of
Serendipity 2-dimensional face elements, following the general guidelines of Section \ref{seregen}. Here, being the first application of our general framework,  the construction is given with much more details than what will be done 
 in the other cases. Section \ref{edge2d} deals with edge 2-dimensional spaces ($H(\rot)$-conforming) and their Serendipity variants. This Section
is very short, since the {\it edge-2d } case can be obtained from the {\it face-2d} case with a simple rotation of $\pi/2$. Section \ref{face3d} deals with the  $H(\div)$-conforming 3-dimensional elements: the first part, with definitions and basic properties of the spaces, based essentially on \cite{super-misti},  is simple and short, while the second part (dealing with the Serendipity variants) is technically more complex. Section \ref{edge3d} deals with $H(\bcurl)$-conforming VEMs, and is possibly the most innovative. The presentation of \cite{super-misti} is generalized and revised, and then we introduce the Serendipity variants, that here,
in addition to {\it internal} degrees of freedom, allow a reduction of the {\it face} degrees of freedom (that could not be dealt with by static condensation). Of the four cases ({\it face} and {\it edge} in 2 and 3 dimensions), the face 3d case is the most complex, but is a useful
step towards the 3d edge case, that in our opinion is the most innovative and interesting one.

\section{Generalities on polynomial spaces}

To denote the independent variables, both in $\R^2$ and in $\R^3$ we will use either $\xx\equiv(x_1,x_2)$ (resp. $\xx\equiv(x_1,x_2,x_3)$ in 3 dimensions) or $(x,y)$ (resp. $(x,y,z)$) whenever this is more convenient. 

In two dimensions, for a scalar function $v$ we define 
\begin{equation} \label{defbrot}
\brot\,v:=\left(\frac{\partial v}{\partial y}, -\frac{\partial v}{\partial x}\right),
\end{equation}
and for a vector $\vv=(v_1,v_2)$ we define the formal adjoint of $\brot$ as
\begin{equation}\label{defrot}
\rot\vv:=\frac{\partial v_2}{\partial x} -\frac{\partial v_1}{\partial y}.
\end{equation}
Always in two dimensions, we recall that  for every $\cc\in\R^2$ we have
\begin{equation} \label{dac2}
\cc=\bgrad(\cc\cdot\xx)\qquad \mbox{as well as}\qquad \cc=\brot(\cc\cdot\xx^{\perp})
\end{equation}
where for a vector $\uu=(u_1,u_2)$ in $\R^2$ its orthogonal $\uu^\perp$ is defined as
\begin{equation} \label{xxperp}
\uu^\perp:=(u_2, -u_1).
\end{equation}
Similarly, in three dimensions for every $\cc\in\R^3$ we have
\begin{equation} \label{dac3}
\cc=\bgrad(\cc\cdot\xx)\qquad \mbox{as well as}\qquad \cc=\frac{1}{2}\bcurl(\cc\wedge\xx).
\end{equation}

Given a polyhedron $\E$, and a smooth-enough vector $\vv$ in $\E$, for every face $f$ with normal $\nn_f$, 
we can define the {\it tangential part of the vector
$\vv$ on $f$} as
\begin{equation}\label{vtauf}
\vtauf:=\vv-(\vv\cdot\nn_f)\nn_f.
\end{equation}
We observe that $\vtauf$ could be obtained from $\vv\wedge\nn^f$ by a suitable rotation of $\pi/2$, so that
\begin{equation}\label{vtaufevwedgen}
\vtauf=0 \qquad \mbox{iff}\qquad\vv\wedge\nn_f=0.
\end{equation}
With (almost) obvious notation, for every face $f$ we could also consider the two-dimensional operators in the tangential
variables $\bgrad_f$, $\div_f$, $\rot_f$, $\brot_f$,  $\Delta_f$, etc. 

\medskip


%
%
%
%
%
%


\subsection{Decompositions of polynomial vector spaces}

On a generic domain $\O$ (in $d$ dimensions, with $d=1,2,\mbox{ or }3$), and for $k$ integer $\ge 0$ we will denote by {\color{black}$\P_{k,d}(\O)$}  (or simply by $\P_{k,d}$ or even $\P_k$ when no confusion can occur) the space of polynomials
of degree $\le k$ on $\O$. With a common notation, we will also use $\P_{-1}\equiv\{0\}$. Following \cite{super-misti} we will denote by $\pi_{k,2}$ the dimension of the space $\P_{k,2}$ (that is, $(k+1)(k+2)/2$), and by $\pi_{k,3}$ the dimension of $\P_{k,3}$ (that is, $(k+1)(k+2)(k+3)/6$).
 Moreover, for $k\ge 1$, we set
\begin{equation}\label{zeromean}
\P^0_k(\O):=\{ p\in \P_k(\O) \mbox{ such that } \int_{\O} p\, {\rm d}\O=0\},
\end{equation}
\begin{equation}\label{divzero}
(\P_k)^d_{div}:=\{\pp \in (\P_k)^d \mbox{ such that } \div \,\pp=0\} .
\end{equation}
{
\noindent The following decompositions of polynomial vector spaces are  well known, and they will be useful in what follows. 
In two dimensions we have
\begin{equation}\label{decompPk2gxp}
(\P_k)^2=\bgrad(\P_{k+1}) \oplus\xx^{\perp}\P_{k-1},
\end{equation}
\begin{equation}\label{decompPk2rx}
(\P_k)^2=\brot(\P_{k+1}) \oplus\xx\P_{k-1}.
\end{equation}
}
\begin{remark}\label{dadecomps2d} A useful consequence of \eqref{decompPk2gxp} is 
the well known property (valid for all $s\ge 0$):
\begin{equation}\label{rotxpq}
\forall\, p_s\in \P_s \,\;\exists \mbox{ a unique}\, q_s\in\P_s\;\mbox{\rm such that }
{\rm rot}(\xx^{\perp}q_s)=p_s.
\end{equation}
The property follows easily from \eqref{decompPk2gxp} with $k=s+1$ by observing that $\rot((\P_{s+1})^2)=
\P_s$.
In proving \eqref{rotxpq} (and the similar properties that follow) we could have used a more constructive argument, but this is simpler. We also notice that an elegant use of the properties of differential operators applied to homogeneous polynomials can be found in \cite{Chin:Lasserre:Sukumar:homo}. We just point out that, as one can easily check, for a {\bf homogeneous}
 polynomial $p_s$ of degree $s$ we have
\begin{equation}\label{homos}
{\rm rot}(\xx^{\perp}p_s)=(s+2)\,p_s.
\end{equation}
Clearly, from
\eqref{decompPk2rx} we have instead
\begin{equation}\label{divxq}
\forall\, p_s\in \P_s \,\;\exists\mbox{ a unique}\, q_s\in\P_s\;\mbox{\rm such that }
{\rm div}(\xx\,q_s)=p_s,
\end{equation}
with  identical arguments.\qed
\end{remark}

%
In three dimensions the analogues of \eqref{decompPk2gxp}-\eqref{decompPk2rx} are
\begin{equation}\label{decompPk3cx}
(\P_k)^3=\bcurl((\P_{k+1})^3) \oplus\xx\P_{k-1},
\end{equation}
\begin{equation}\label{decompPk3gxw}
(\P_k)^3=\bgrad(\P_{k+1}) \oplus\xx\wedge(\P_{k-1})^3.
\end{equation}

\begin{remark} In computing the dimension of the space $\xx\wedge(\P_{k-1})^3$
that appears in \eqref{decompPk3gxw}, it has to be noted that $\xx\wedge(\xx\,\P_s)\equiv 0$ for all $s$, so that the dimension of $\xx\wedge(\P_{k-1})^3$ is equal to that of $(\P_{k-1})^3$ minus the dimension 
of $ \P_{k-2}$. And, indeed, one can check that $3\pi_{k,3}=[\pi_{k+1,3}-1]+[3\pi_{k-1,3}-\pi_{k-2,3}]$. In its turn, taking into account that the dimension of $\bcurl((\P_{k+1})^3)$ is equal to the dimension of $(\P_{k+1})^3$ minus that of $\bgrad(\P_{k+2})$, we can check the dimensions in \eqref{decompPk3cx}
through $3\pi_{k,3}=[3\pi_{k+1,3}-\{\pi_{k+2,3}-1\}]+\pi_{k-1,3}$. One should just avoid mistakes in the math. \qed
\end{remark}
\begin{remark}\label{dadecomps3d} As in Remark \ref{dadecomps2d},
useful consequences of \eqref{decompPk3cx} and \eqref{decompPk3gxw} are
the equally  well known properties (valid for all $s\ge 0$):
\begin{equation}\label{divxq3}
\forall\, p_s\in \P_s \,\;\;\exists\, \qq_s\in(\P_s)^3\;\mbox{\rm such that }
{\rm div}(\xx\,q_s)=p_s,
\end{equation}
and
\begin{equation}\label{rotxpq3}
\forall\, \pp_s\in (\P_s )^3 ~\mbox{\rm with }\div\pp_s=0\,\;\;\exists\, \qq_s\in(\P_s)^3
~\mbox{\rm with }\div\qq_s=0\;\mbox{\rm such that }
\end{equation}
$${\bf curl}(\xx\wedge \qq_s)=\pp_s.$$
\end{remark}
\qed

\subsection{Polynomial Spaces}

In the Mixed Finite Elements practice one typically encounters vector valued polynomial spaces of a special type. We recall some of them.
For $k\ge 0$, in $2$ or $3$ dimensions, we have
\begin{equation}\label{deftRT2}
RT_k:=(\P_k)^d\oplus \xx\,\P_k^{hom}
\end{equation}
(where, here and in the sequel, the superscript {\it hom} stands for {\it homogeneous}) and, for $k\ge 1$,
\begin{equation}\label{deftBDM2}
BDM_k:=(\P_k)^d.
\end{equation}
 It is simple but useful to note that in any case
\begin{equation}
\{\vv\in RT_k\}\,\mbox{and }\{\div\vv=0\} \mbox{ imply }\{\vv\in (\P_k)^d\}.
\end{equation}

These two types of elements are tailored for the construction of $H(\div)$-conforming mixed finite elements on simplexes. Typically the {\bf normal} components (on edges in 2d, and on faces in 3d) are used as boundary degrees of freedom, so that their continuity, from one element to another, will ensure the $H(\div)$ conformity of the global space. The difference between the two families is that, for a given accuracy $\P_k$ of the normal components
at the boundary,  we have $\div(BDM_k)=\P_{k-1}$  and $\div(RT_k)=\P_k$, so that the {\it RT} elements are
recommended when you need a better accuracy in $H(\div)$, while the {\it BDM} elements are cheaper for the same
accuracy in $L^2$. They are both quite popular and widely used.

The $H(\rot)$ (in 2d) or $H(\bcurl)$ (in 3d) counterparts of these elements are the N\'ed\'elec elements
of first type ({\it N1}) and  of second type ({\it N2}). In two dimensions, the two types are just the $RT$ and 
(respectively) \textit{BDM} elements, up to a rotation of $\pi/2$:
\begin{equation}\label{deftN1-2D}
N1_k:=(\P_k)^2\oplus \xx^{\perp}\,\P_k^{hom}
\end{equation}
\begin{equation}\label{deftN2-2D}
N2_k:=(\P_k)^2.
\end{equation}
The differences (between {\it RT} and \textit{BDM}, on
one side, and {\it N1}-{\it N2} on the other side) are much more relevant in 3d. Indeed, in 3d we have
\begin{equation}\label{deftN1-3D}
{\color{black}N1_k:=(\P_k)^3\oplus \xx\wedge\,(\P_k^{hom})^3}
\end{equation}
\begin{equation}\label{deftN2-3D}
N2_k:=(\P_k)^3.
\end{equation}
Here the  {\bf tangential} components
at the boundary have to be prescribed to ensure the $H(\rot)$-conformity.  
This is done by assigning
 the tangential component on each edge, and then completing the set of degrees of freedom, per face,
with the internal ones.

The above spaces are very well suited for applications to simplicial elements. When applied, in 2d, on squares  (and their affine or isoparametric images)
their definition changes. For instance, on rectangles the spaces $RT$ become
\begin{equation}\label{defqRT2}
RT^q_k:=\Q_{k+1,k}\times\Q_{k,k+1}
\end{equation}
where for integers $r$ and $s$ we used the common notation: 
\begin{equation} \label{Qrs}
\Q_{r,s}=\{\mbox{polinomials of degree $\le r$ in $x_1$ and of degree $\le s$ in $x_2$}\} 
\end{equation}
while
\begin{equation}\label{defqBDM2-q1}
BDM^q_k:=(\P_k)^2\oplus{\rm span}\{\brot(x^{k+1}y)\}\oplus{\rm span}\{\brot(xy^{k+1})\}.
\end{equation}
In  3d, for cubes  we have
\begin{equation}\label{defqRT3}
RT^q_k:=\Q_{k+1,k,k}\times\Q_{k,k+1,k}\times\Q_{k,k,k+1}
\end{equation}
with obvious extension of the notation \eqref{Qrs}.
The definition of \textit{BDM} on cubes is more complicated (see e.g. \cite{A-A}, \cite{Bo-Bre-For}).

{\color{black}We also point out that the non affine images of  these spaces for boxes (squares or cubes) exhibit several forms of severe approximation deficits.}



\section{General strategy towards Serendipity Spaces}\label{seregen}
We position ourselves {\it at the current element level} of a decomposition, and we consider a very general type of local spaces. Then, as usual, the local spaces will be put together to construct the global Virtual Element spaces defined on the whole computational domain.

Let then $\E$ be a polytope, in two or three dimensions, and let $\VV$ be a finite dimensional space  made of smooth enough functions. Let $\NN$ be the dimension of $\VV$; 
we assume that we have $\NN$ linear functionals $\FF_1,...,\FF_ {\NN}$ from $\VV$ to $\R$, linearly independent,
that play the role of {\it original degrees of freedom}.

The {\it name of the game} is to be able to {\it slim down} the space $\VV$ and, accordingly, the degrees of freedom $\FF$, in such a way that we preserve certain properties at a cheaper price.

\subsection{The d.o.f. and the subspace  that we want to keep}\label{keepdof}

We then assume that, among the original degrees of freedom, we have a subset of degrees of freedom that
we want to keep. Typically these will be the boundary ones (or a subset of them,  necessary to
ensure the desired conformity properties for the global space) plus, possibly, some internal ones that we
will need in order to satisfy some additional properties (for instance, an {\it inf-sup} condition of the global space
with respect to some other given space). 

All this will become clear in the examples that follow, but for 
the moment we do not need to specify them. We just assume that our original
degrees of freedom are numbered in such a way that those that we want to keep are the first ones. In other words, given an integer number $\MM\le \NN$, the degrees of freedom that we want to keep are $\FF_1,...,\FF_ {\MM}$.

We also assume that we have a subspace $\SS\subset\VV$
that we want to preserve while reducing $\VV$. A typical example would be to choose $\SS$
as the space of polynomials up to a certain degree, that we want to keep in order to ensure the desired accuracy for the
final discretized problem.

\subsection{The crucial step}

Here comes the crucial step: we assume that we have an intermediate set of degrees of freedom (or, with a suitable
numbering of the original ones, an integer $\SN$ with $\MM\le\SN\le\NN$ ) having the  {\it crucial property} defined here below.

\begin{defn} 
The degrees of freedom $\FF_1,...,\FF_{\SN}$ are {\bf $\SS$-identifying} if
\begin{equation}\label{Sident}
\forall q\in \SS \qquad \{\FF_1(q)=...=\FF_{\SN}(q)=0\}\Rightarrow\{q\equiv 0\}.
\end{equation}
\end{defn}
\qed


\noindent Defining the operator $\DD_{\SN}:\SS\rightarrow\R^{\SN}$ by
\begin{equation}\label{defDD}
\DD_{\SN}(q):=(\FF_1(q),...,\FF_{\SN}(q)),
\end{equation} 
we immediately have that property \eqref{Sident} could also be expressed as:
$$\mbox{ {\it $\DD_{\SN}$ is injective from $\SS$ to $\R^{\SN}$}}.$$

\begin{remark} It is clear that we have, actually, to choose the degrees of freedom
(within $\FF_{\MM+1},...,\FF_{\NN}$ ) that we want to add, and then re-order
the degrees of freedom so that the first $\SN$ are ``the first $\MM$ ones plus the 
chosen additional ones''. However, quite often in what follows, with an abuse
of language we will talk about {\bf choosing $\SN$} to mean that we
choose the additional degrees of freedom and, if necessary, we re-order the whole set.\qed
\end{remark}
Note that, in a certain number of cases, we will be allowed to take $\SN=\MM$, meaning that the degrees
of freedom ``that we want to keep in any case'' are already $\SS$-identifying. In other cases, we will have to add other
degrees of freedom, on top of the first $\MM$ ones, in order to have \eqref{Sident}.
\begin{remark}In all the examples in this paper,  the choice of the degrees of freedom that we want to keep, and the choice of
the space $\SS$ that we want to preserve will be dictated by general needs on the properties of ``the global
space that comes out of the local spaces used within each element'': conformity, accuracy, compatibility with 
other spaces, and so on. On the other hand, the choice of the additional  $\FF_{\MM+1},...,\FF_{\SN}$ degrees of freedom
(if any) will depend very much on several other properties, related to the combination of: the shape of $\E$,
the space $\SS$, and the degrees of freedom $\FF_1,...,\FF_{\MM}$.\qed
\end{remark}

In all cases, the first important step will be to 
check whether the initial 
$\FF_1,...,\FF_{\MM}$ are {\it already} $\SS$-identifying or not. And
if they are not, an even more crucial (and sometimes delicate) step will be to identify the space
\begin{equation}\label{defZG}
\Z:=\{ q\in\SS\mbox{ such that } \FF_1(q)=...=\FF_{\MM}(q)=0\}
\end{equation}
and decide what are the additional degrees of freedom needed to obtain \eqref{Sident}.

In several cases, depending on the dimension (2 or 3), on the types of spaces (nodal, edge, face), on the degree  of the polynomials $\SS$, and on the geometry of the element
we are working on, the identification of $\Z$, and the identification
of a possible set of additional degrees of freedom 
\begin{equation}\label{addit}
\FF_{\MM+1},...,\FF_{\SN}
\end{equation}
will be relatively easy, and computationally cheap. In other cases, it risks to be a nightmare. It is therefore worthwhile, in our opinion, to introduce a {\it general strategy} that, though rather expensive (in terms of operations to be performed at the element level), can be used
in a systematic (and conceptually simple) way in the computer code. As we shall see,
in cases where the same decomposition is going to be used many times (for solving PDE problems with different coefficients, or with different right-hand sides) such a procedure,
implemented once and for all, could be of great help.

\label{system}\subsection{A systematic way to pick $\FF_{\MM+1},...,\FF_{\SN}$} To further simplify the presentation, we also assume that the degrees of freedom $\FF_{\MM+1},...,\FF_{\NN}$ are naturally 
{\it sliced} in several layers: typically, when they correspond to moments against a polynomial space, the slices could be the
homogeneous polynomials of increasing degree: $0, 1, 2,..., k$, or some obvious adaptations of this same slicing to other cases, for instance when $\SS$ is, say, a Raviart-Thomas space or a N\'ed\'elec-first kind space. Note that this is {\it not necessary}
(we could always take $\NN$-$\MM$ slices of size 1), but it could help in simplifying the code, as well as the intuitive grasp of the procedure. Hence we introduce the integer numbers $\sigma_0,\sigma_1,...\sigma_\rho$ to identify the slices:
\begin{equation}\label{slices}
\{\FF_{\MM+1},...,\FF_{\MM+\sigma_0}\},\{\FF_{\MM+\sigma_0+1},...,\FF_{\MM+\sigma_1}\}, ..., \{\FF_{\MM+\sigma_\rho+1},...,\FF_{\NN}\} .
\end{equation}
Let $\NSS$ be the dimension of $\SS$. Taking a basis $s_1,...,s_{\NSS}$ in $\SS$  we can therefore consider the 
$\NSS \times\NN$ matrix $D$ given by
\begin{equation}\label{defD}
D_{ij}:=\FF_j(s_i).
\end{equation}
Since $\SS\subseteq\VV$, and the degrees of freedom $\FF_{1},...,\FF_{\NN}$
are unisolvent in $\VV$ we easily have that the matrix $D$ has maximum rank
(i.e. rank equal to $\NSS$). Note that to say that $\Z\equiv\{{\bf 0}\}$ is
equivalent to say that the sub-matrix $D_{\MM}$, made by the first $\MM$ 
columns of $D$, has already maximum rank. And our target (in choosing $\NSS$) is to have
a sub-matrix  $D_{\NSS}$ (made by the first $\NSS$ columns of $D$) that has
maximum rank. Having to choose $\SN$ we can proceed (in a sort of {\it brutal} way)
by checking successively the sub-matrices
\begin{equation}\label{sottom}
D_{\MM}, D_{\MM+\sigma_0}, D_{\MM+\sigma_1},...
\end{equation}
until we find the first one that has maximum rank (that surely exists, since the whole matrix $D\equiv D_{\NN}$ has maximum rank). This will determine a viable choice
for $\SN$.  

Needless to say, in a number of particular cases we could find a simpler, cheaper, and
sometimes more effective way of choosing $\SN$, as we shall see in the following sections. However, the general strategy described above has to be considered as a solid {\it back-up}
that allows us to proceed even in the worst cases. Hence in the remaining part of the present section, that deals with the {\it general strategy} to
construct our Serendipity-like spaces, we shall assume, from now on, that $\SN$ has
been chosen.

\subsection{Construction of the Serendipity subspaces}

It is now time to explain the way to construct (after $\SN$ has been chosen)  our Serendipity-like local spaces, and in particular to see
how property \eqref{Sident} is used for it. 

Assume therefore that we have chosen the degrees of freedom $\FF_1,...,\FF_{\SN}$, and let us construct a suitable
(serendipity!) subspace $\VV_{\SN}$, with $\SS\subseteq\VV_{\SN}\subseteq\VV$, for which  $\FF_1,...,\FF_{\SN}$
are a unisolvent set of degrees of freedom.

The procedure will now be simple, since we prepared everything already.

We define a projection operator $\RR^{\SN}:\VV\rightarrow\SS$ as follows. For $v\in\VV$, we define $\RR^{\SN}v$ 
as the {\it unique} element of $\SS$ such that
\begin{equation}\label{defRS}
\Big[\DD_{\SN}(\RR^{\SN}v),\DD_{\SN}q\Big]_{\R^{\SN}}=\Big[\DD_{\SN} v,\DD_{\SN}q\Big]_{\R^{\SN}}\quad\forall q\in\SS
\end{equation}
where $[\,\cdot~,\,\cdot]_{\R^{\SN}}$ is the Euclidean scalar product in $\R^{\SN}$. Note that the fact that
$\DD_{\SN}$ is injective (that is, property \eqref{Sident}) plays a crucial role in ensuring that problem \eqref{defRS}
has a unique solution in $\SS$. Needless to say, the  Euclidean scalar product $[\,\cdot~,\,\cdot]_{\R^{\SN}}$
could be substituted by any other symmetric and positive definite bilinear form on $\R^{\SN}$.

Once $\RR^{\SN}$ has been defined, we can introduce the serendipity space $\VV_{\SN}$ as
\begin{equation}\label{defVS}
\VV_{\SN}:=\{v\in\VV\mbox{ such that }\FF_i(v)=\FF_i(\RR^{\SN}v)\;(i=\SN+1,...,\NN)\}.
\end{equation}
The following proposition is an immediate consequence of this construction.
\begin{prop}
 With the above construction, the degrees of freedom  $\FF_1,...,\FF_{\SN}$ are unisolvent for the space 
$\VV_{\SN}$. Moreover, if $v\in\VV_S$, using $\FF_1(v),...,\FF_{\SN}(v)$ one can compute the remaining $\FF_{\SN+1}(v),...,\FF_{\NN}(v)$. Finally, we observe that $\SS \subseteq \VV_{\SN}$.\qed
\end{prop}

To summarize the results of the present section, we recall that, in all cases, in order to pass from the
original space (with original degrees of freedom) to the Serendipity space (with a smaller number of degrees
of freedom), one has to:
\begin{itemize}
\item Identify the $\MM$ degrees of freedom that we want to keep, and the polynomial space $\SS$ that we want 
to maintain inside the local space.
 \item  Consider the space $\Z$ defined in \eqref{defZG}.
\item If $\Z\equiv\{\bf 0\}$, take $\SN=\MM$ and proceed directly to \eqref{defRS}, and then to \eqref{defVS}.
\item If instead $\Z$ contains some nonzero elements, identify $\NN_{\Z}$ 
additional degrees of freedom that, added to the previous $\MM$, form a set of $\SS$-identifying degrees of freedom, in the sense of Definition \eqref{Sident}. Then take $S=\MM+\NN_{\Z}$.
\end{itemize}
Clearly, in the latter case, $\NN_{\Z}$ will have to be equal, or bigger than the dimension of $\Z$.  

In the following sections we will first recall the mixed virtual element spaces already introduced in \cite{super-misti}
(although with slightly different degrees of freedom), and then discuss the application of the general Serendipity strategy to each particular case.



\section{Face Virtual Elements in 2d}\label{face2d}

\subsection{Definition of the VEM spaces}
We start by considering the  two-dimensional face elements  $\bV^f_{k,k_d,k_r}(\E)$.
For $k$, $k_d$, $k_r$ integers, with $k\ge 0$, $k_d\ge 0$, $k_r$  $\ge -1$ we set:
\begin{equation}\label{VEMf2d}
\bV^f_{k,k_d,k_r}(\E)\!:=\!\{\vv|\, \vv\cdot\nn_e\!\in\!\P_k(e) \,\forall \mbox{ edge } e,\, \div\vv\!\in\!\P_{k_d}(\E),\,\rot\vv\!\in\!\P_{k_r}(\E)\},
\end{equation}
with the following degrees of freedom: 
\begin{align}
&D_1:\quad\int_e\vv\cdot\nn_e\,q_k\de\quad\mbox{ for all }q_{k}\in\P_k(e), \mbox{ for all edge }e,\label{f2d_d0f4}
\\
&D_2:\quad\mbox{ for $k_d\ge 1$: }\int_{\E}\vv\cdot\bgrad q_{k_d}\dE \quad\mbox{ for all $q_{k_d}\in\P_{k_d}(\E)$},
\label{f2d_d0f5}
\\
&D_3:\quad\mbox{ for $k_r\ge 0$: }\int_{\E}\vv\cdot\xx^\perp\,q_{k_r}\dE \quad\mbox{ for all $q_{k_r}\in\P_{k_r}(\E)$}.\label{f2d_d0f6}
\end{align}

\begin{prop}\label{unidoff2d}
The degrees of freedom \eqref{f2d_d0f4}-\eqref{f2d_d0f6} 
 are unisolvent. Moreover, they allow to compute the 
$L^2(\E)$-orthogonal projection operator from  $\bV^f_{k,k_d,k_r}(\E)$ to $(\P_s)^2$ for every integer $s\le k_r+1$.
\end{prop}

\begin{proof}
First we observe that, using \eqref{rotxpq} with
$s=k_r$,  for every $\vv$ in $\bV^f_{k,k_d,k_r}(\E)$ we can always find a $p_{k_r}$ such that $\rot(\xx^\perp\,p_{k_r})=\rot\vv$. Then  $\rot (\vv-\xx^\perp\,p_{k_r})=0$, and therefore $\vv-\xx^\perp\,p_{k_r}$ is a gradient. We deduce that:
\begin{equation}\label{decoe2d}
\left\{\aligned&\mbox{\it every $\vv\in\bV^f_{k,k_d,k_r}(\E)$ can be written in a unique way as}\\
&\vv=\bgrad\phi+\xx^{\perp}p_{k_r}\quad\mbox{\it  for some function $\phi$ and some $p_{k_r}$ in $\P_{k_r}$}.
\endaligned
\right.
\end{equation}
This immediately gives the unisolvence of the d.o.f.  Indeed, the number of d.o.f. being equal to the dimension of $\bV^f_{k,k_d,k_r}(\E)$, we have to show that if   a $\vv$ in $\bV^f_{k,k_d,k_r}(\E)$  verifies $D_1$=$D_2$=$D_3$=$0$, then $\vv\equiv 0$. From $D_1$=$0$ we immediately deduce $\vv\cdot\nn=0\mbox{ on }\partial\E$ which, together with $D_2$=$0$ and an integration by parts gives $\div\,\vv=0\mbox{ in }\E.$ Consequently:
\begin{equation}\label{perpaigrad2-p}
\int_{\E}\vv\cdot\bgrad\varphi\dE = -\int_{\E}\div\vv\,\varphi\dE+\int_{\partial\E}\vv\cdot\nn\,\varphi\de=0\quad\forall\varphi\in H^1(\E).
\end{equation}
Finally, using \eqref{decoe2d}, then \eqref{perpaigrad2-p} and $D_3=0$:
\begin{equation}\label{unis2df}
\int_{\E}|\vv|^2\dE=\int_{\E}\vv\cdot(\bgrad\phi+\xx^{\perp}p_{k_r})\dE\\= 0 + 0.
\end{equation}
Arguing as  for \eqref{perpaigrad2-p}  we see that the d.o.f. \eqref{f2d_d0f4} and \eqref{f2d_d0f5} allow to compute  the integral
$\int_{\E}\vv\cdot\pp\dE$ for every $\pp=\bgrad \vphi$, and $\vphi$ polynomial of any degree. On the other hand,  the d.o.f. \eqref{f2d_d0f6} 
allow to compute also $\int_{\E}\vv\cdot\xx^{\perp}p\dE$ for every $p\in\P_{k_r}$.  Hence, looking now at  \eqref{decompPk2gxp}, we deduce that for every $s\le k_r+1$ and for every $\vv\in V^{f}_{k,k_d,k_k}(\E)$ we can compute the $L^2$-projection $\Pi^0_{s}\vv$ on $(\P_s(\E))^2$.

\end{proof}

 \begin{remark} 
 \label{deponly}
 As it comes out clearly from the last part of the above proof,  once the degrees of freedom \eqref{f2d_d0f4} and \eqref{f2d_d0f5} match the values
 of $k$ and $k_d$ (respectively) in \eqref{VEMf2d}, then {\bf the computability of  $\Pi_s^0$} (for $s$ arbitrarily big) {\bf depends  only on the value of $k_r$}. \hfill  \qed
 \end{remark}

 \begin{remark}\label{freedom} It is easy to see that when used in combination with the degrees of freedom \eqref{f2d_d0f4}, the degrees
 of freedom \eqref{f2d_d0f5} can equivalently be replaced by
 \begin{equation}\label{f2d_d0f5-alt}
\bullet\mbox{ for $k_d\ge 1$: }\int_{\E}\div\vv\,q_{k_d}\dE \quad\mbox{ for all $q_{k_d}\in\P_{k_d}^{\,0}(\E)$}. \qed
\end{equation}
\end{remark}
Along the same lines, it should also be pointed out, at the general level, that for the same space we could obviously construct a huge number
of different unisolvent sets of d.o.f. which, as such, are all equivalent. In some cases the procedure that one 
has to follow to pass from one set to an equivalent one is reasonably simple and can be performed  with
a modest amount of computations. In other cases, however, this passage  would require much more difficult computations: typically, the solution of a partial differential equation (or even a system of partial differential equations) within the element, something that goes far beyond the work that one is ready to perform.
Here for instance, instead of the degrees of freedom \eqref{f2d_d0f6} we could clearly use 
\begin{equation}\label{f2d_d0f6-alt}
\bullet\mbox{ for $k_r\ge 0$: }\int_{\E}\rot\vv\,q_{k_r}\dE \quad\mbox{ for all $q_{k_r}\in\P_{k_r}(\E)$}.
\end{equation}
It is however easy to see that in order to pass from one set to the other we should solve a $\div-\rot$ system in $\E$.
Depending on what you need to compute inside the element $\E$ you must therefore choose and use one set of degrees
of freedom, and forget the other ones that are equivalent but not ``computationally equivalent''.

 \begin{remark} \label{menouno} In principle, one could consider, say, face Virtual Elements  with $k_d=-1$, implying that we restrict our attention to divergence-free vectors. Unfortunately, in this case, the divergence theorem requires $\int_{\partial\E}\vv\cdot\nn=0$ in 
 \eqref{f2d_d0f4}, and we could not use a local basis in the computational domain.
\qed
 \end{remark}
 
\subsection{Comparisons with Finite Elements}

 The comparison between VEMs and FEMs can only be done on a limited number of classical geometries
(here for simplicity we consider only simplexes and boxes). However it should be clear from the very beginning that VEMs allow much more general geometries. For these more general geometries the comparison should actually be done between VEMs and  other methods designed for polytopes, as for instance \cite{MFD23}, \cite{Bishop}, \cite{Cangiani:Georgoulis:Houston2014}, \cite{Cockburn-IMU},  \cite{DiPietro-Ern-Lemaire}, \cite{Floater:Acta}, \cite{Gillette-2}, \cite{Fries:Belytschko:2010}, \cite{MFD22}, \cite{GFEM146}, \cite{Gillette-14}, \cite{Wachspress2016}.


 \begin{figure}[!]
  \begin{center}
    \includegraphics[width=8.0cm]{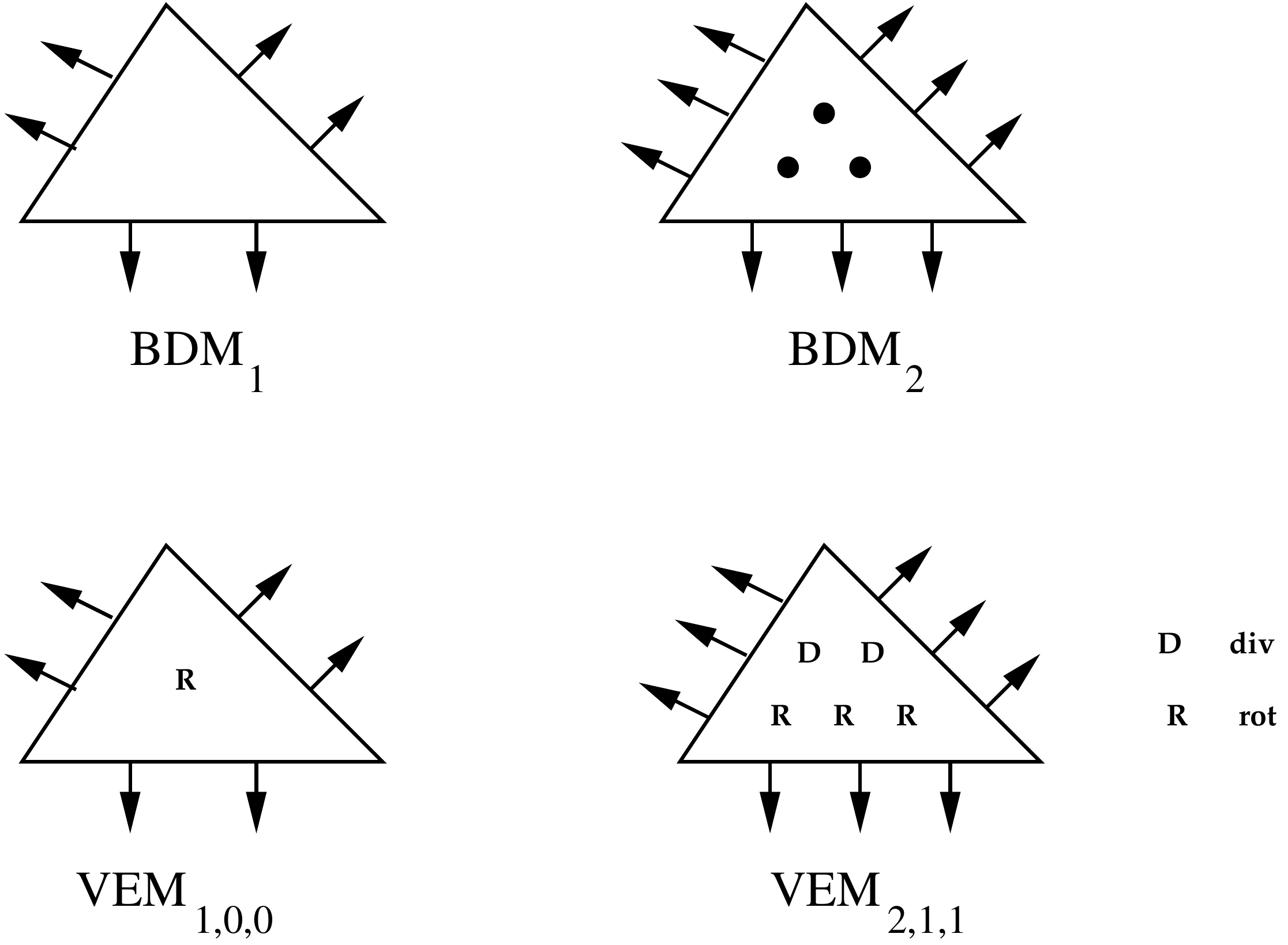}
  \end{center}
  \vskip-0.3truecm
  {\caption{Triangles: $BDM_k$ and $VEM_{k,k-1,k-1}$}\label{clas-tri-f}}
  \vskip0.3truecm
\end{figure}

The natural comparison, within Finite Elements, of our $V^f_{k,k-1,k-1}$ elements are clearly the \textit{BDM} spaces as described in \eqref{deftBDM2} for triangles (see Figure \ref{clas-tri-f}).


The same comparison for quadrilaterals is shown in  Figure \ref{clas-triBDM}. In both cases we see that the elements in $V^f_{k,k-1,k-1}$  have a higher number of degrees of freedom than the corresponding \textit{BDM} Finite Elements.

\begin{figure}[!]
  \begin{center}
    \includegraphics[width=8.0cm]{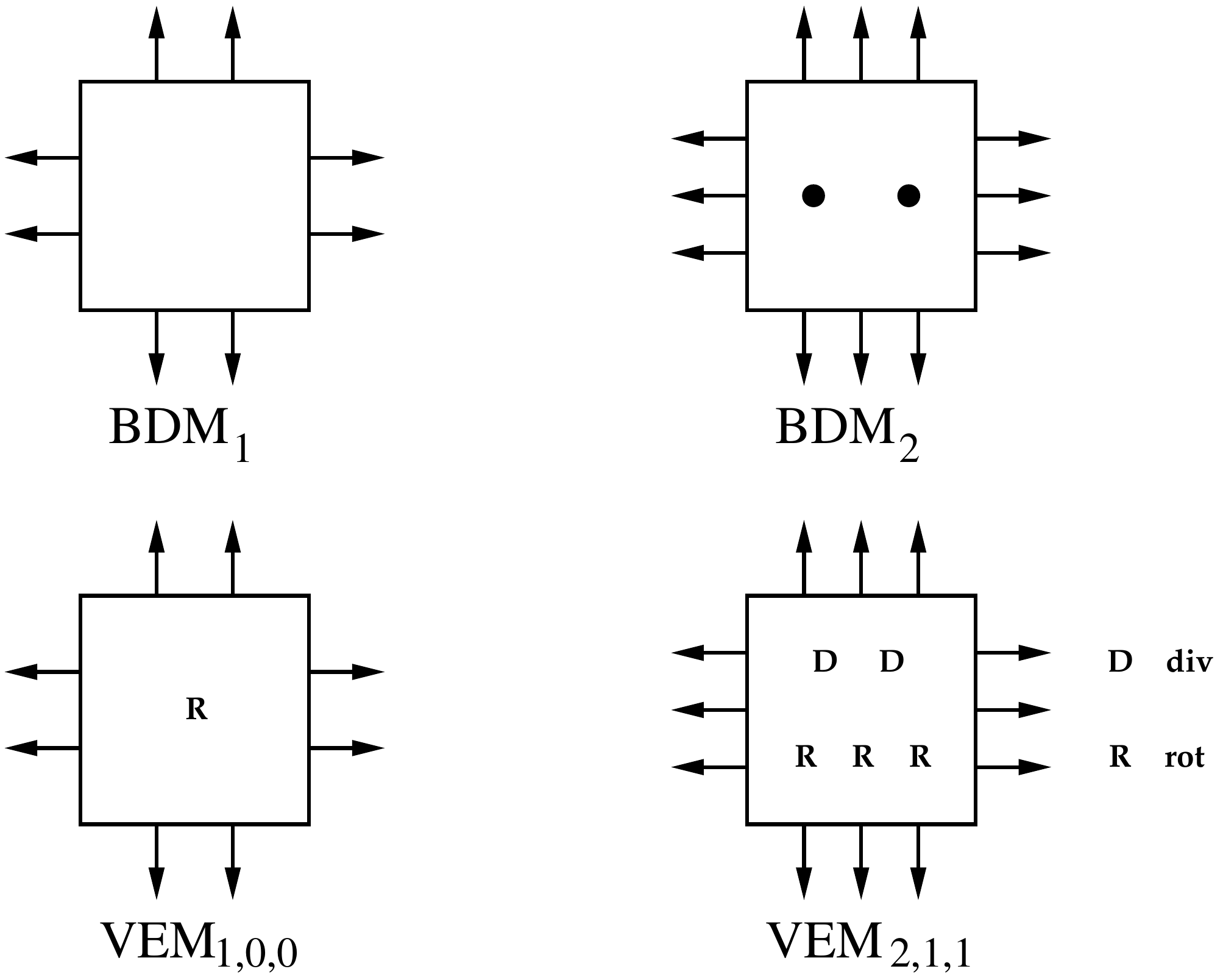}
  \end{center}
  \vskip0.7truecm
  {\caption{Quadrilaterals: \textit{BDM}$_k$ and $VEM_{k,k-1,k-1}$}\label{clas-triBDM}}
  \vskip0.7truecm
\end{figure}
On the other hand, the natural counterpart for $V^f_{k,k,k-1}$ are the classical Raviart-Thomas
elements. For comparison, see   Figure \ref{clas-tr-0} for triangles,
\begin{figure}[ht!]
  \begin{center}
    \includegraphics[width=9.6cm]{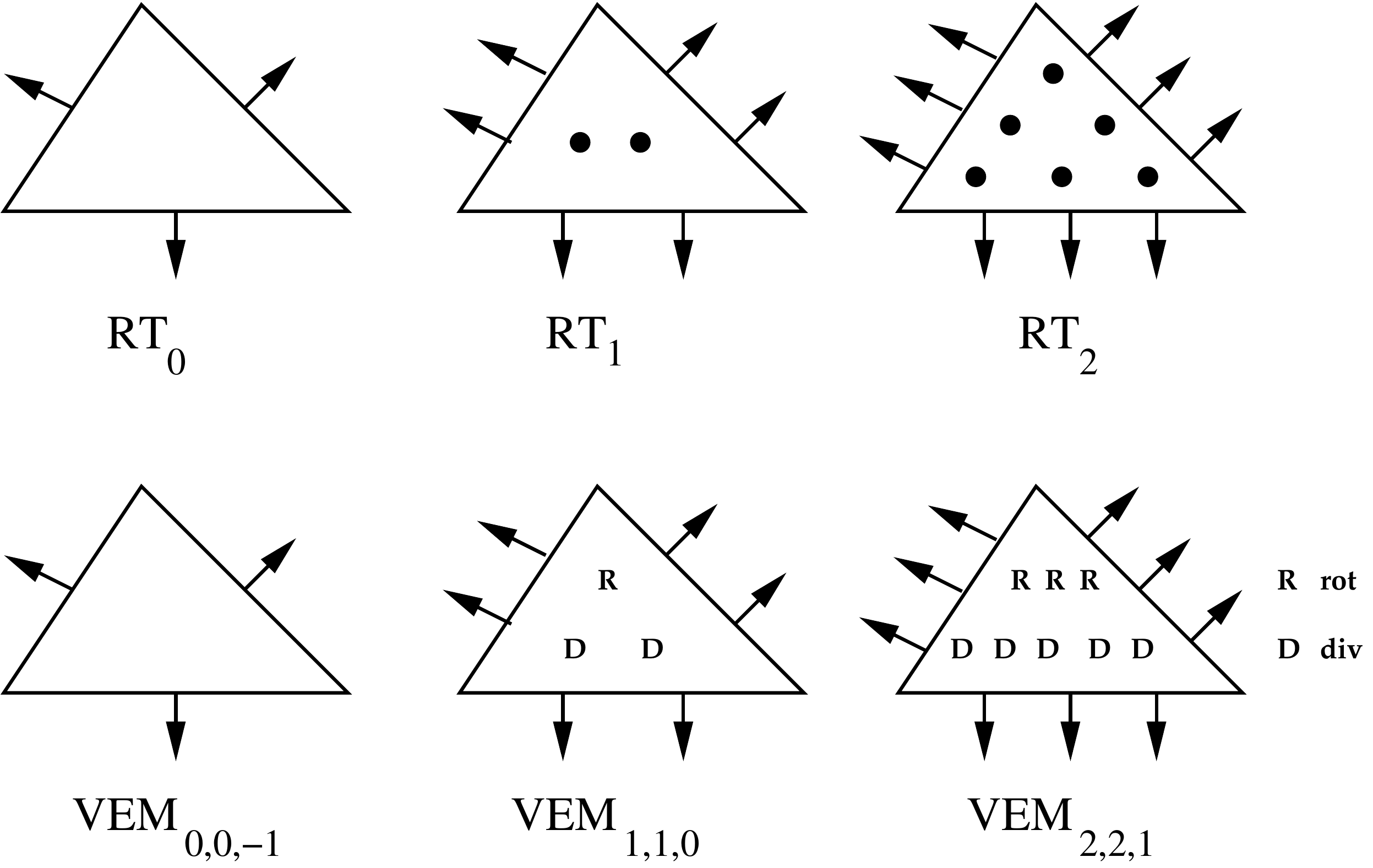}
  \end{center}
  \vskip-0.3truecm
  {\caption{Triangles: $RT_k$ and $VEM_{k,k,k-1}$}\label{clas-tr-0}}
  \vskip0.3truecm
\end{figure} 
\noindent where again VEMs have more degrees of freedom. Instead, on quadrilaterals VEMs have a smaller number
of degrees of freedom than $RT$ (see Figure \ref{clas-qua1}).
 \begin{figure}[ht!]
  \begin{center}
    \includegraphics[width=9.6cm]{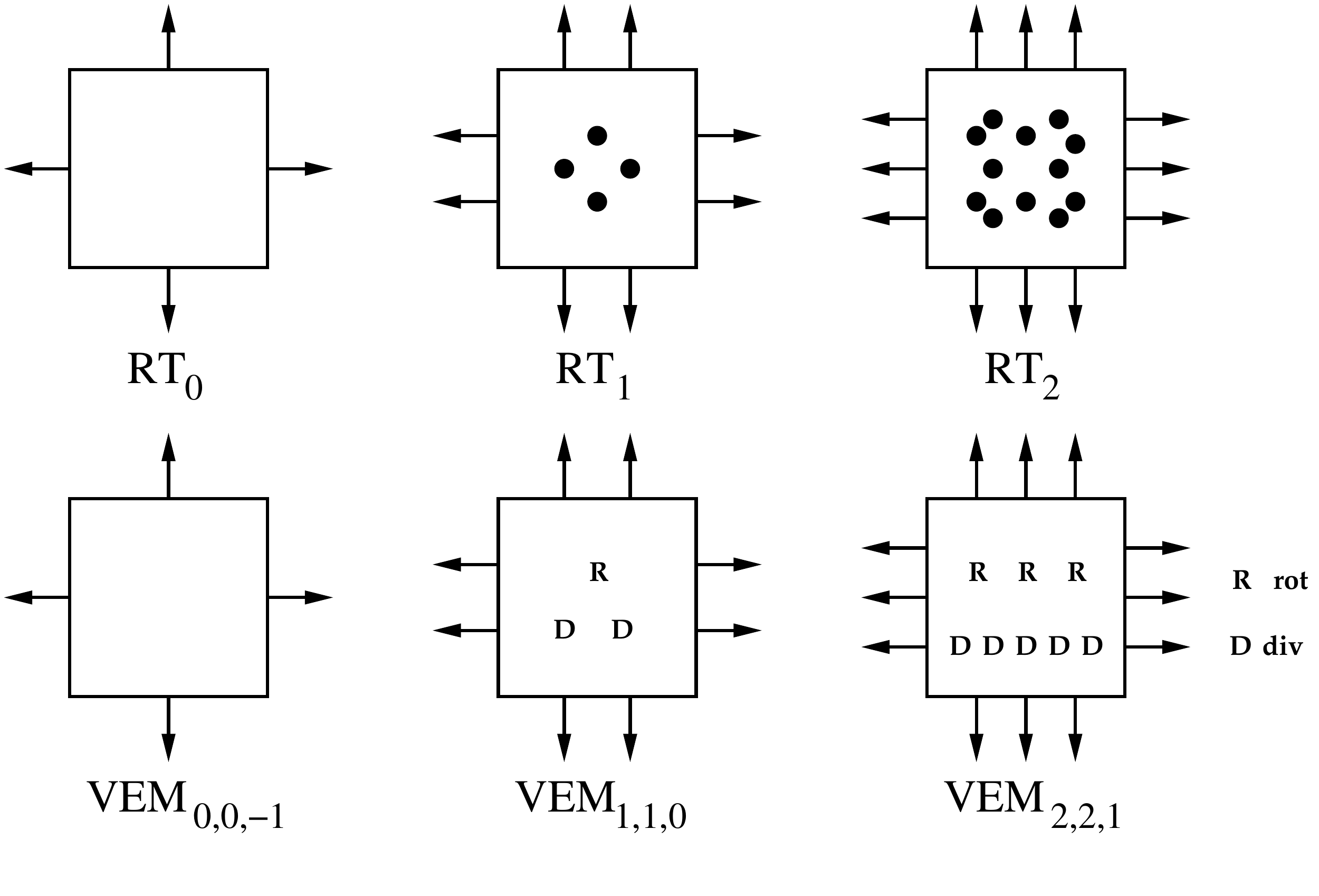}
  \end{center}
  \vskip-0.7truecm
  {\caption{Quadrilaterals: $RT_k$ and $VEM_{k,k,k-1}$}\label{clas-qua1}}
  \vskip0.7truecm
\end{figure}

Finally, the natural counterpart of the $VEM^f_{k,k,k}$  are the {\it BDFM} finite element spaces. We omit a detailed comparison, and we only point out that here too VEMs have more degrees of freedom.
%
%



\subsection{B-Compatible Interpolators for Face VEM in 2d}

Given a smooth enough vector valued function $\uu$, we can now use the degrees of freedom \eqref{f2d_d0f4}--\eqref{f2d_d0f6} to define an interpolation operator that, for brevity, we
denote by $\Pi^I$ (neglecting its obvious dependence on $k$, $k_d$, $k_r$, and $\E$) given by
\begin{align}
\label{f2d_d0fI}
&\bullet\;\; \Pi^I\uu\in\bV^f_{k,k_d,k_r}(\E) \qquad\mbox{and}\\
\label{f2d_d0f4I}
&\bullet\int_e(\uu-\Pi^I\uu)\cdot\nn_e\,q_k\de=0\quad\mbox{ for all $q_{k}\in\P_k(e)$, for all edge $e$,}\\
\label{f2d_d0f5I}
&\bullet\mbox{ for $k_d\ge 1$: }\int_{\E}(\uu-\Pi^I\uu)\cdot\bgrad q_{k_d}\dE =0\quad\mbox{ for all $q_{k_d}\in\P_{k_d}(\E)$},\\
\label{f2d_d0f6I}
&\bullet\mbox{ for $k_r\ge 0$: }\int_{\E}(\uu-\Pi^I\uu)\cdot\xx^\perp\,q_{k_r}\dE=0 \quad\mbox{ for all $q_{k_r}\in\P_{k_r}(\E)$}.
\end{align}
%
It is easy to check that $\Pi^I$ is a B-compatible operator (in the sense, for instance, of Section 5.4.3
of \cite{Bo-Bre-For}). In our particular case, this means that
\begin{equation}\label{Bcompatf2}
\int_{\E}\div(\uu-\Pi^I\uu)q_{k_d}\dE=0,\qquad\forall q_{k_d}\in\P_{k_d}(\E),
\end{equation}
which is an easy consequence of \eqref{f2d_d0f4I} and \eqref{f2d_d0f5I} upon an integration by parts.

%

\begin{remark}\label{XBcompat}
It is important to point out that in the definition \eqref{f2d_d0f4I}-\eqref{f2d_d0f6I} of the
operator $\Pi^I$, only the degrees of freedom \eqref{f2d_d0f4I}-\eqref{f2d_d0f5I} are necessary in order to have \eqref{Bcompatf2}. Hence, among the degrees of freedom that we will want to keep (in our Serendipity approach), we will  have to include \eqref{f2d_d0f4I}-\eqref{f2d_d0f5I}, in order to preserve conformity and B-Compatibility, while
the degrees of freedom \eqref{f2d_d0f6I} will be, so to speak, ``expendible''. \qed
\end{remark}

\section{Serendipity face elements in 2 dimensions}\label{sereface}

We want now to eliminate some of the internal degrees of freedom of the VEM spaces defined in the previous section, 
following the general strategy of Sect. \ref{seregen}. As we have seen there, we have to decide first what are the $\MM$ degrees of freedom that we want to keep, and what is the polynomial space that we want to preserve. 

The first choice (concerning the degrees of freedom) is rather simple, as we already pointed out in Remark \ref{XBcompat}: in order to have an $H(\div)$-conforming global space we need to keep all the boundary degrees of freedom, i.e., \eqref{f2d_d0f4} in the present case;  and in order to preserve the {\it B-compatibility}  we also need the degrees of freedom \eqref{f2d_d0f5}. Concerning the space to be preserved, the obvious choice would be $\SS=BDM_k\equiv (\P_k)^2$ if  $k_d=k-1$, and $\SS=RT_k$ if $k_d=k$. Clearly, these are not the only possible reasonable choices. In particular cases, other choices could also be valuable. For instance, if we know that the $H(\div)$ component of the
 solution of our problem is a gradient, we can restrict out attention to the case of the gradients of $\P_{k+1}$
 (as suggested, for instance, in \cite{BFM}. See also \cite{MFD23} in the context of Mimetic Finite Differences).
 
 Here however, we don't want to enter the details of a very general setting. Hence, we will limit ourselves to the cases
 $k_d=k$ and $k_d=k-1$, that, as we shall see, can be treated with the same arguments. For this, we will denote simply by $\SS_k$ the space to be preserved, knowing that it should be either $BDM_k$ or $RT_k$. Still following 
 Section \ref{seregen}, we go then  hunting for the space $\Z$ in \eqref{defZG} that in {\it both} our cases
 reduces to
 \begin{equation}
 \Z_k:=\{\vv\in (\P_k)^2 \mbox{ such that } \div\vv=0  \mbox { in } \E \mbox{ and }\vv\cdot\nn=0\mbox { on }
 \partial\E\} .
\end{equation}
Assuming for simplicity that $\E$ is simply connected, $\Z_k$ can also be written as
 \begin{equation}\label{defZk}
 \Z_k=\brot \Big( \P_{k+1}\cap H^1_0(\E)\Big) .
\end{equation}
\subsection{Characterization of $\Z_k$}
Following \cite{SERE-Nod} we  start by observing that, for $r$ integer $\ge 1$, if a polynomial ${\color{black}p_r\in\P_{r,2}}$  
vanishes identically on a segment (with positive measure) belonging to the straight line with equation
$ax+by+c=0$, then $p_r$  can be written in the form
\begin{equation}
p_r(x,y)=(ax+by+c)\,q_{r-1}(x,y)
\end{equation}
for some polynomial $q_{r-1}$ of degree $r-1$. As a consequence, if a polynomial {\color{black}$p_r$} in $\P_{r,2}$ vanishes identically
on $r+1$ segments (with positive measure) belonging to $r+1$ {\it different} straight lines, then $p_r$ is identically zero. So far so good. Now, to the polygon $\E$ we attach the integer number $\eta_\E$ defined as
\begin{equation}\label{defeta} 
\eta_{\E}:= \mbox{\it the minimum number of  straight lines necessary to cover $\partial \E$,} 
\end{equation}
and we recall the following obvious but useful property (already used in \cite{SERE-Nod}).
\begin{prop}\label{Z}
 Let $p_r\in\P_{r,2}$ be a polynomial of degree $r$ that vanishes identically on  $\partial\E$. Then
 for $r<\eta_E$ we have  $p_r\equiv 0$, and for $r\ge\eta_\E$ we have that $p_r$ must be of the form
$p_r=q_{r-\eta_\E}b_{\eta}$,  where $q_{r-\eta_\E}$ is a polynomial of degree $r-\eta_\E$ and $b_{\eta}$ is a polynomial of degree $\eta_\E$ that vanishes identically on $\partial\E$.\qed
\end{prop}
As an immediate consequence of this and of \eqref{defZk}, we have that 
 \begin{equation}\label{defZZk}
 \Z_k:=\left\{\begin{array}{lr}
 \{{\bf 0}\} &\text{  for } \eta_\E>k,\\
 \brot \Big(b_{\eta_\E}\P_{k-\eta_\E+1}\Big)&\text{  for } \eta_\E\le k.
 \end{array}\right.
\end{equation}
\begin{remark}\label{changesign} If $\E$ is convex then $b_\eta$ will not change sign in $\E$, a property that will become handy in just a while. Moreover, assume that $\E$ is not convex, but there are only two ``re-entrant'' edges  (more precisely: edges belonging to straight lines that intersect the interior of $\E$, and consequently whose equations change sign inside $\E$), and let $\gamma_2$ be the second degree polynomial that vanishes on the  two straight lines that contain the two re-entrant edges. In this case it is easy to see that the product $b_\eta\,\gamma_2$ does not change sign
in $\E$. \qed
\end{remark}

The following Lemma is an immediate consequence of Proposition \ref{Z}.
\begin{lem}\label{addoff2d} 
Assume, for simplicity, that $\E$ is convex, and 
let $(k, k_d, k_r)$  be a triplet of integers with $k\ge 0$, $k_d\ge \max\{0, k-1\}$ and $k_r\ge k+1-\eta_\E$.
Assume that $\pp_k\in \Z_k$ is  such that
\begin{equation}\label{f2d_d0f6eq0}
\int_{\E}\pp_k\cdot(\xx^{\perp}q_{s})\dE=0 \quad\mbox{ for all $q_{s}\in\P_{k+1-{\eta_\E}}(\E)$}.
\end{equation}
Then $\pp_k\equiv {\bf 0}$.
\end{lem}

\begin{proof}
Using \eqref{defZZk} we have that if $k+1<\eta_\E $ the proof is immediate, while  for $k +1\ge\eta_\E$ then $\pp_k=\brot(b_{\eta}\psi)$ for some polynomial $\psi$ of degree $k+1-\eta_\E$. Then we use \eqref{rotxpq} with $s=k+1-\eta_E$ and $p_s=\psi$  to get a $q_s$ such that  $\rot(\xx^\perp\,q_s)=\psi$, 
and we insert it in \eqref{f2d_d0f6eq0} to obtain
\begin{multline}\label{f2d_d0f6eq01}\!
 0\!=\!\!\int_{\E}\pp_k\cdot\,\xx^{\perp} q_s\dE\!=\!\!\int_{\E}\brot(b_{\eta}\,\psi)\cdot\,\xx^{\perp} q_s\dE\\
\!=\!\!\int_{\E}(b_{\eta}\,\psi)\rot\,(\xx^{\perp} q_s)\dE\!= \!\!\int_{\E}b_{\eta}\psi^2\dE
\end{multline}
that ends the proof since $b_\eta$ does not change sign.
\end{proof}

\begin{remark}\label{noconvex} If we give up the convexity assumption we  could always follow the path of Subsection \ref{system}.  Otherwise, we  should find some ``ad hoc'' alternative ways to design suitable sets of conditions that, in a way similar to Lemma \ref{addoff2d}, imply that $\pp_k={\bf 0}$. This is surely possible in many circumstances. For instance, assume that $\E$ is a quadrilateral with two re-entrant edges, and $k=3$ (so that $k+1-\eta_\E=0$,  and \eqref{f2d_d0f6eq0}  would be required just for $q_s$ constant). Assuming that the origin is in the re-entrant vertex, we could use, instead of \eqref{f2d_d0f6eq0},
\vskip-0.5truecm
\begin{equation}
\int_{\E}\pp_k\cdot(\xx^{\perp}\gamma_2)\dE=0,
\end{equation}
\vskip-0.1truecm
\noindent where $\gamma_2$ is ``the product of the two re-entrant edges'' as in Remark \ref{changesign}.
 It is immediate to see that, as the origin is in the re-entrant vertex, then $\gamma_2$ is a {\bf homogeneous}
 polynomial of degree 2, so that from \eqref{homos} we have $\rot(\xx^\perp\,\gamma_2)=4\gamma_2$. Hence, still following  Remark \ref{changesign}, we have that $b_\eta\gamma_2$ does not change sign,  and therefore the argument in \eqref{f2d_d0f6eq01} still
goes through. Indeed, always for $k+1=\eta_\E$ we would have now that $\varphi=\lambda b_\eta$ for some constant $\lambda$,
and then:
\vskip-0.6truecm
\begin{multline}\label{f2d_d0f6eq011}
\hskip1.2truecm 0=\int_{\E}\pp_k\cdot(\xx^{\perp} \gamma_2)\dE=\int_{\E}\brot(\lambda\, b_{\eta})\cdot(\xx^{\perp} \gamma_2)\dE=\\
=\int_{\E}(\lambda\, b_{\eta})\,{\rot}\,(\xx^{\perp} \gamma_2)\dE=4\int_{\E}\lambda\, b_{\eta}\gamma_2 \dE
\end{multline}
that implies $\lambda=0$ and ends the proof.
However, a detailed study of the different cases of non convex polygons and of the possible remedies goes beyonds the scope of the present paper,
and in any case we always have the systematic path of Subsection \ref{system}. Throughout the sequel of the paper, for simplicity, we will stick to the convexity assumption. \qed
\end{remark}

 \subsection{The Serendipity face spaces}
 
 At this point we just have to follow the general setting of Section \ref{seregen}: define as in \eqref{defDD} the mapping $\DD_{\SN}$ of the degrees of freedom, use it to define the operator $\RR^{\SN}$ as in \eqref{defRS}, and finally define our serendipity space as in \eqref{defVS}.
 
 \begin{remark}\label{BconS}
 It is easy to see that even for our Serendipity spaces we can construct an interpolation operator,
 using this time the  degrees of freedom \eqref{f2d_d0f4} and 
 \eqref{f2d_d0f5}, plus those in \eqref{f2d_d0f6eq0} when $\eta_\E\ge k$. It is also easy to see that such an interpolation operator will be B-compatible.   \qed
 \end{remark}

The new Serendipity elements  can again be compared, for triangular and quadrilateral domains, with classical finite elements of different types. The comparison with {\bf triangular} elements is, in some sense, not very interesting, since (as it can be easily checked) the new Serendipity Virtual Elements  coincide now exactly with the classical (polynomial) Finite Elements, having the same number of degrees of freedom, and being one included in the other. On the other hand, on {\bf quadrilaterals} we have now a considerable gain, as it can be seen in Figure \ref{clas-tri-2}.
 \begin{figure}[htbp]
  \begin{center}
    \includegraphics[width=14.0cm]{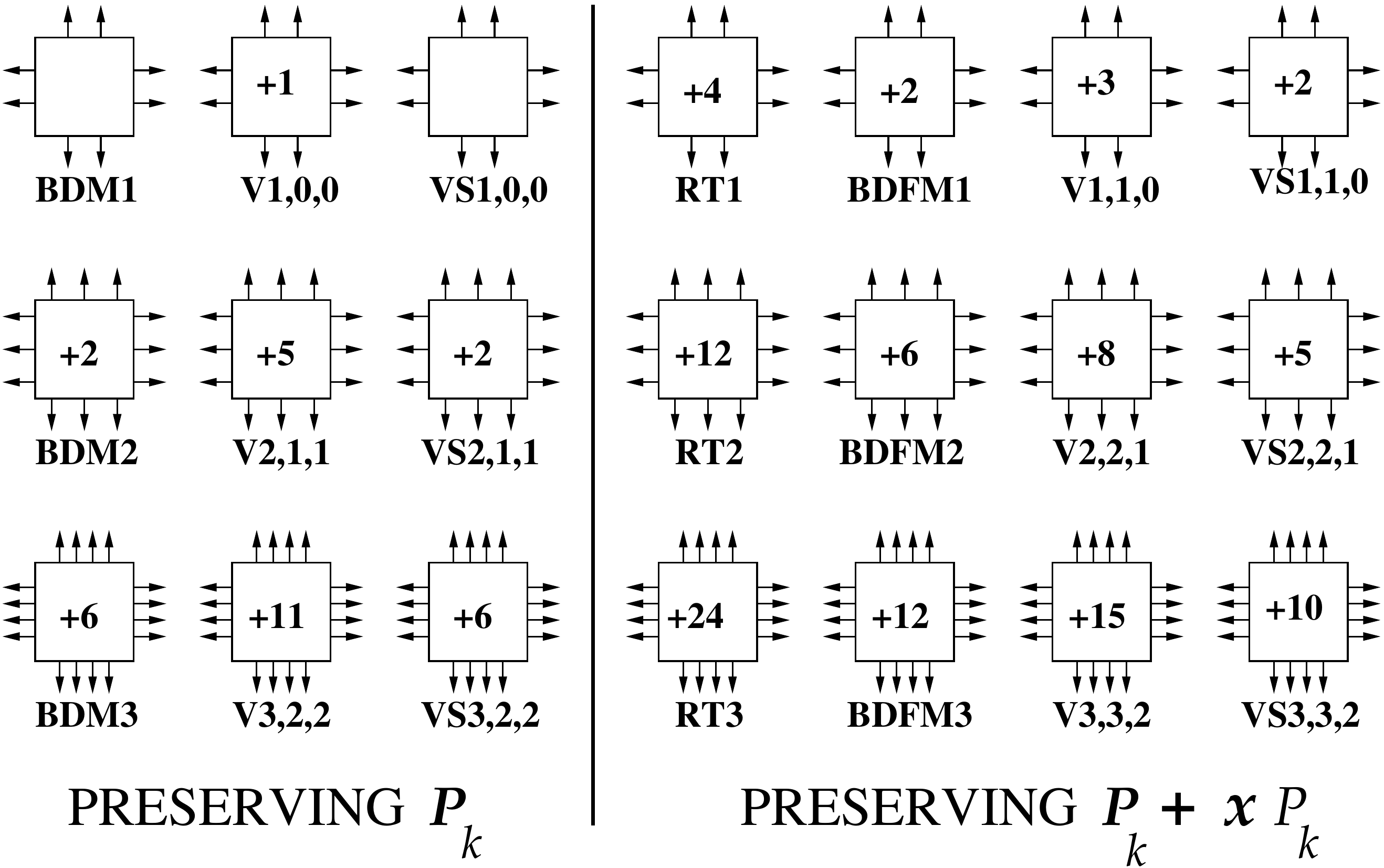}
  \end{center}
  \vskip-0.3truecm
  {\caption{FEM spaces, VEM spaces  and  Serendipity ones}\label{clas-tri-2}}
  \vskip0.3truecm
\end{figure}
%
%
%
%
In particular we can see that the new $VEMS^f_{k,k-1,k-1}$ (that is, Serendipity with $\SS_k=BDM_k$) have the same number of degrees of freedom as the corresponding \textit{BDM} spaces (although it has also to be noted  that on the one hand VEMs are much more robust with respect to geometric distortions, but on the other hand the elimination of the internal degrees of freedom require an additional work that is not present in the traditional Finite Elements).  Instead, the $VEMS^f_{k,k,k-1}$ (that is, the Serendipity VEMs with $\SS_k=RT_k$) have now much  less degrees of freedom than the corresponding Finite Element  {\it RT} spaces. And we recall once more that VEMs are defined on almost arbitrary geometries. The comparison, actually, should be done with Serendipity 
{\it RT} spaces. In that case we have exactly that same number of degrees of freedom as (for instance)
the elements in \cite{A-A} (but again, with much more generality in the geometry and additional work inside the elements).


\begin{remark}\label{angle} For stability reasons, in practice, in the definition \eqref{defeta} of $\eta_E$   it would be wise to apply the same (slight) correction that we used already in \cite{SERE-Nod} for nodal elements: it consists in taking a
smaller value of $\eta_\E$ (and hence using more degrees of freedom) whenever we have two or more edges that belong {\bf almost} to
the same straight line. Practically this corresponds to decide (once and for all) a minimum angle $\theta_0$ and then
to consider that two straight lines are ``distinct'' only  if they cross with four angles all bigger than $\theta_0$. Parallel lines
can be accepted if their distance is not too small (compared with the diameter of $\E$).
Note that, in the framework of the general systematic strategy of Subsection \ref{system},
this would correspond to decide the minimum amount of the smallest singular value in the matrix $\calD$ in \eqref{sottom} to be accepted in order to say that it has ``maximum rank''. \qed

\end{remark}

 \begin{remark}\label{stability} Always for stability reasons, the use of the Euclidean scalar product in $\R^S$ in \eqref{defRS} is recommended
 only if the degrees of freedom ``scale in the same way'' (a concept widely used in the VEM context: see e.g. \cite{volley}). \qed
 \end{remark}

 \subsection{The lazy choice, the stingy choice, and the static condensation} Always following what has been done in   \cite{SERE-Nod},
 we can distinguish different types of strategies to be adopted in coding these elements, in particular when dealing with very general geometries. The two extremes of this set of possible
 choices have been called  {\bf the stingy choice} and {\bf the lazy choice.} Here we recall the
 basic ideas behind them, pointing out, however, that there a number of intermediate strategies, to be used to adapt to the different situations.
 The {\bf stingy choice} corresponds to use the Serendipity strategy in order to {\it drop as many internal degrees of freedom as we can}. This, referring to Remark \ref{angle}, and considering for simplicity the case of convex polygons, would correspond to compute $\eta_E$ taking a small minimum angle, and then reduce the number of internal
 additional degrees of freedom to $\pi_{s,2}$ with $s\le k-\eta_\E+1$. In a more general context (even without the convexity assumption), following  the general strategy of Subsection \ref{system}, this would imply, for instance, to take slices of dimension 1,  and discard each one that does not increase the rank of the submatrix
$\calD_{\MM+\sigma}$ in \eqref{sottom}.
 
 The {\bf lazy choice}, instead, would correspond to {\it minimize the work necessary to choose the additional d.o.f. \eqref{addit}}. This can be done, for instance, pretending that $\eta_\E=3$, and therefore 
 considering, as additional degrees of freedom \eqref{f2d_d0f6eq01}, all the polynomials of the form
 $\xx^\perp q$ with $q\in\P_{k-2} $. Note that already on a general quadrilateral mesh, for any quadrilateral that is not degenerated into a triangle
 our theory allows to take in \eqref{addoff2d} $q\in\P_{k-3}$,  using only $\pi_{k-3,2}$ additional degrees of freedom and saving $k-1$ (that is:
 $\pi_{k-2,2}-\pi_{k-3,2}$) degrees of freedom with respect to the triangular case. But, as a counterpart, it would require to check the ``non-degeneracy into a triangle''  of every
 element. A non-obvious trade-off. For more general decompositions, with a high $k$, both the cost and the gain of the stingy choice would  be more conspicuous. Then the decision could rely on several factors, including the degree $k$ but also, for instance, the number of problems that we plan to solve on the same grid.
 
 Another matter that would be worth discussing is the comparison with {\it static condensation} techniques,
 that, when solving with a sophisticated direct method, could become almost automatic and be  reasonably cheap. There too,
 the gain/loss assessment is not always obvious. We just point out that the present serendipity procedure is {\bf not} equal to static condensation (as, for instance, the static condensation of the internal node of a $\Q_2$ nine-node finite element is not equal to use a Serendipity eight-node element). Moreover we point out that, if Serendipity elements are used on  the faces of a three-dimensional decomposition, then the gain  is much more clear, since the static condensation of face unknowns is surely far from obvious. 
 







  


  


  

 
 
 
 
\section{General edge elements in 2d}\label{edge2d}

\subsection{Edge VEM spaces and  degrees of freedom}\label{edge2d-1} 
The case of {\it edge} elements  in two dimensions can be treated exactly as we did for {\it face} elements. We summarize them quickly. We set, 
for every $k\ge 0$, $k_d\ge -1$, and $k_r\ge 0$:
\begin{equation}\label{VEMe2d}
\bV^e_{k,k_d,k_r}(\E):=\{\vv|\vv\cdot\tt_e\!\in\!\P_k(e) \forall \mbox{ edge } e,\, \div\vv\! \in\!\P_{k_d}(\E),\,\rot\vv\!\in\!\P_{k_r}(\E)\},
\end{equation}
with the degrees of freedom: 
\begin{align}
&\widetilde D_1:\quad\int_e\vv\cdot\tt_e\,q_k\de\quad\mbox{ for all $q_{k}\in\P_k(e)$, for all edge $e$,}
\label{e2d_d0f4}\\
&\widetilde D_2:\quad\mbox{ for $k_r\ge 1$: }\int_{\E}\vv\cdot\brot q_{k_r}\dE \quad\mbox{ for all $q_{k_r}\in\P_{k_r}(\E)$},
\label{e2d_d0f5}\\
&\widetilde D_3:\quad\mbox{ for $k_d\ge 0$: }\int_{\E}\vv\cdot\xx\,q_{k_d}\dE \quad\mbox{ for all $q_{k_d}\in\P_{k_d}(\E)$}.
\label{e2d_d0f6}
\end{align}
%
%
Similar to Proposition \ref{unidoff2d} we have
\begin{prop}\label{unidofe2d}
The degrees of freedom \eqref{e2d_d0f4}-\eqref{e2d_d0f6} are unisolvent.
\end{prop}
Moreover, proceeding as  in Remark \ref{deponly} we easily see that out of the above d.o.f. one can compute
\begin{equation}\label{prokr2d}
\int_{\E}\vv\cdot\qq\dE\quad\mbox{\rm for every $\qq\in(\P_{k_d+1})^2$},
\end{equation}
 and hence the projection operator on $(\P_{k_d+1}(\E))^2$.
\begin{remark} It is easy to see that when used in combination with the degrees of freedom \eqref{e2d_d0f4}, the degrees
 of freedom \eqref{e2d_d0f5} can equivalently be replaced by
 \begin{equation}\label{f2d_d0f5-altr}
\bullet\mbox{ for $k_r\ge 1$: }\int_{\E}\rot\vv\,q_{k_r}\dE \quad\mbox{ for all $q_{k_r}\in\P_{k_r}^0(\E)$}.
\end{equation}
Here too we could argue as in Remark \ref{freedom} regarding other equivalent but possibly not ``computationally equivalent''
degrees of freedom. \qed
\end{remark}

\begin{remark} In almost all applications, the value of $k_r$ in \eqref{VEMe2d} is either equal to $k$ or equal to $k-1$. This, as we already saw for face elements, corresponds to choices mimicking the N\'ed\'elec Finite element spaces of  first and second
kind (that is, {\it N1} and {\it N2}), and, ultimately, the choice among the two cases depends on the accuracy that we demand in $H(\rot)$ (and not only in $L^2$). \qed
\end{remark}

\begin{remark} As we did in Remark \ref{menouno}, for Edge Virtual Elements we
cannot take $k_r=-1$, unless we give up the possibility of having a local basis. \qed
\end{remark}

\subsection{Edge Serendipity  VEMs in 2d}\label{sere-edge-2D}

We can now extend all the definitions and results obtained for {\bf Face Serendipity VEMs} to the case of  {\bf Edge Serendipity VEMs},
just by changing, as we did so far, ``$\nn$''  into ``$\tt$'',  then ``$\div$'' into ``$\rot$'', and finally ``$k_d,\, k_r$'' into ``$k_r,\, k_d$''. 
In particular we have now
 \begin{equation}
 \Z_k:=\{\vv\in (\P_k)^2 \mbox{ such that } \rot\vv=0  \mbox { in } \E \mbox{ and }\vv\cdot\tt=0\mbox { on }
 \partial\E\} .
\end{equation}
Assuming for simplicity that $\E$ is simply connected, $\Z_k$ can also be written as
 \begin{equation}\label{defZkg}
 \Z_k:=\bgrad \Big( \P_{k+1}\cap H^1_0(\E)\Big)
\end{equation}
that can be analyzed exactly as in the case of face VEM. Recalling Proposition \ref{Z} we have now
\begin{equation}\label{defZZke}
 \Z_k:=\left\{\begin{array}{lr}
 \{{\bf 0}\}, &\text{  for } \eta_\E>k,\\
 \bgrad \Big(b_{\eta_\E}\P_{k-\eta_\E+1}\Big)&\text{  for } \eta_\E\le k.
 \end{array}\right.
\end{equation}
Then for our Serendipity space
we just have to keep \eqref{e2d_d0f4}-\eqref{e2d_d0f5} plus, for $k+1-\eta_\E\ge 0$, the additional  d.o.f. 
\begin{equation}\label{f2d_SERE3-add}
\hskip-1.7truecm\bullet\;\int_{\E}\vv\cdot\xx\,q \dE \quad\mbox{ for all $q\in\P_{k+1-\eta_E}(\E)$}.
\end{equation}
Then everything proceeds as a mirror image of what has been done and said in Sections \ref{seregen} and \ref{sereface}. In particular, after choosing $\SS$ as (tipically) $N1_k$ or $N2_k$, we can construct a projector $\RR^{\SN}:\bV^e_{k,k_d,k_r}(\E)\rightarrow\SS$, based on the degrees of freedom
\eqref{e2d_d0f4}, \eqref{e2d_d0f5}, and   \eqref{f2d_SERE3-add}, and then define, as
in \eqref{defVS},
\begin{equation}\label{defVSe2d}
\bV^e_{k,S,k_r}(\E):=\{v\in\bV^e_{k,k_d,k_r}(\E)\mbox{ s. t. }\FF_i(v)\!=\!\FF_i(\RR^{\SN}v),\,i\!=\!\SN+1,..,\NN\}.
\end{equation}


  
  


%
%
%

%
%
%
%
%
%
%
%
%
%
%
%
%
%
%
%
%
%
%
%
%
%
%
%
%
%
%
%
%
%
%

 \section{General face elements in 3d}\label{face3d}
 
\subsection{The spaces and the degrees of freedom}
The definition of the spaces $\bV^f_{k,k_d,k_r}$ in three dimensions is an immediate generalization 
of the two-dimensional case, essentially using \eqref{decompPk3gxw} instead of \eqref{decompPk2gxp}.  For  $k\ge 0$, $k_d\ge 0$, and $k_r\ge -1$ they can be defined as
\begin{multline}\label{VEMf3d}
\bV^f_{k,k_d,k_r}(\E):=\{\vv|\, \mbox{ such that }\vv\cdot\nn_f\in\P_k(f) \,\forall \mbox{ face } f,\; \\\div\vv\in\P_{k_d}(\E),\,\bcurl\,\vv\in(\P_{k_r}(\E))^3\}.
\end{multline}

 It is easy to see (arguing as in the two-dimensional case) that we can take, as degrees of freedom in $\bV^f_{k,k_d,k_r}(\E)$, the following ones
 \begin{align}
 \label{f3d_dof1}
& \bullet \int_f\vv\cdot\nn_f\,q_k\df\quad\mbox{ for all $q_{k}\in\P_k(f)$, for all face $f$},
\\
&\bullet\mbox{ for  }k_d\ge 1: \quad
\label{f3d_dof2kd}
 \int_{\E}\vv\cdot\bgrad q_{k_d}\dE \quad\mbox{ for all $q_{k_d}\in\P_{k_d}(\E)$},
\\
&\bullet \mbox{ and for }k_r\ge 0:\quad 
\label{f3d_dof2kr}
 \int_{\E}\vv\cdot\xx\wedge\,\qq_{k_r}\dE \quad\mbox{ for all $\qq_{k_r}\in (\P_{k_r}(\E))^3$}.
\end{align}
It is also easy to see that, proceeding as in the proof of Proposition \ref{unidoff2d}, out of the above degrees of freedom one can compute the integral
\begin{equation}\label{prokr3d}
\int_{\E}\vv\cdot\qq\dE\mbox{ for every $\qq\in(\P_{k_r+1})^3$},
\end{equation}
and then the $L^2$-projection operator $\Pi^0_{k_r+1}$ on the space $(\P_{k_r+1,3})^3$.

\begin{remark} As we did in the previous cases, we can easily see that we could substitute the degrees of freedom 
\eqref{f3d_dof2kd} with the equivalent ones
\begin{equation}\label{f3d_dof2kd-alt}
\hskip-0.4truecm\bullet \int_{\E}\div\vv\, q_{k_d}\dE \quad\mbox{ for all $q_{k_d}\in\P_{k_d}^0(\E)$}.
\end{equation}
It is also immediate to see that the degrees of freedom \eqref{f3d_dof2kd-alt} are also ``computationally equivalent''
to \eqref{f3d_dof2kd}. \qed

\end{remark}

\begin{remark} \label{freedom-f3D}In different applications, one could give up the possibility to compute the $L^2$-projection operator $\Pi^0_{k_r+1}$ and use, instead of \eqref{f3d_dof2kr}, the  degrees of freedom
\begin{equation}\label{f3d_dof2kr-alt}
\hskip-0.0truecm\bullet \int_{\E}\bcurl\,\vv\cdot\qq_{k_r}\dE \quad\mbox{ for all $\qq_{k_r}\in (\P_{k_r}(\E))^3_{div}$}.
\end{equation}
(see the notation \eqref{divzero})
that are equivalent (but not ``computationally equivalent'', in the spirit of Remark \ref{freedom}) to \eqref{f3d_dof2kr}. \qed

\end{remark}

\subsection{{\color{black}Serendipity face elements in 3d}}
The construction of the Serendipity variants of the face Virtual Elements  defined in \eqref{VEMf3d}  is decidedly more complicated than in the two-dimensional case. As before, in order to have an $H(\textrm{div})$ conforming global space and preserve the B-compatibility, we will need to keep the degrees of freedom \eqref{f3d_dof1} and \eqref{f3d_dof2kd},  so that the Serendipity reduction
will act  only on the degrees of freedom \eqref{f3d_dof2kr}.

But the main difference here 
is in the characterization of the space $\Z_k$:
\begin{equation}\label{defZ3}
\Z_k(\E)=\{\zz\in(\P_k)^3\mbox{ such that }\zz\cdot\nn=0\mbox{ on }\partial\E,\mbox{ and }
\div\zz=0\mbox{ in }\E\}.
\end{equation}
Indeed, in two dimensions, the elements of $\Z_k$ were
the $\brot$ of a {\it scalar function} vanishing on the whole $\partial\E$, and their characterization (in Proposition \eqref{Z}) was relatively easy. In three dimensions, instead, we
have the $\bcurl$ of a {\it vector valued} potential whose tangential components vanish on all faces.

It is immediate to see that for $k=0$ and $k=1$ the space $\Z_k$ is reduced to $\{{\bf 0}\}$.  The characterization of the elements of $\Z_k$ for $k\ge 2$, instead, is less obvious, and to perform it for a general polyhedral geometry looks rather heavy and complex, so that it seems
advisable to stick on the systematic strategy of Subsection \ref{system}, unless the decomposition has some particular feature that could be exploited.

 However, to give an idea of the type of problems to be tackled, we report here, as an example, 
the treatment of the simplest case of a tetrahedral element (that might also come out handy if
we opt for some kind of {\it lazy choice}).

\subsection{The space $\Z_k(\E)$ for $k\ge 2$ on a tetrahedron}

We need some additional notation.  
Let $f_1,..,f_4$ be the faces of $\E$, let $\lambda_i$ be {\it the} $\P_1$
polynomial such that: $\lambda_i(\xx)=0$ is the plane containing the face $f_i$, and 
 the outward unit normal to $\E$ on $f_i$ is given by $\nn_i=\nabla\lambda_i$. Then, 
\begin{itemize}
\item let $b_4$ be the fourth degree polynomial $\lambda_1 \lambda_2\lambda_3\lambda_4$,
\item for $i=1,..,4$ let  $b_3^{(-i)}$ be the product of 
all the $\lambda_j$ with $j\neq i$.
\end{itemize}
We also recall the elementary equality
\begin{equation}\label{rotdelprod}
\bcurl(\vphi\nabla\psi)=\nabla \vphi\wedge\nabla\psi\qquad \forall \vphi, \psi \in H^1.
\end{equation}

\begin{prop}\label{prope1} With the above notation, for every polynomial $p$  and for every face $f_i$ (with  $1\le i\le 4$)
we have
\begin{itemize}
\item $(\nabla p\wedge \nabla \lambda_i)\cdot \nn_i=0$
\item $\div(\nabla p\wedge\nabla \lambda_i)=0$  
\item For every $i$ with $1\le i\le 4$, if $p$ contains $b_3^{(-i)}$ as a factor,  then 
we have $(\nabla p\wedge \nabla \lambda_i)\cdot\nn=0$ on all $\partial\E$. 
\end{itemize}
\end{prop}
\begin{proof}
The first statement follows immediately from the fact that $\nn_i=\nabla\lambda_i$ and the properties of the scalar triple product. The second follows immediately from \eqref{rotdelprod} observing that  $\nabla p\wedge \nabla \lambda_i=\bcurl(p\nabla\lambda_i)$. Finally, to see the third we remark that on each face $f_j~ (j\neq i)$
the condition ``$p=0$ on $f_j$'' implies that $\nabla p$ is directed as $\nn_j=\nabla\lambda_j$. Hence, on
each face $f_j$ (whether $j=i$ or not!) at least one between $\nabla p$ and $\nabla \lambda_i$ is directed as $\nn_j$ and the scalar triple product vanishes.
\end{proof}

As a consequence of Proposition \ref{prope1} we have that: 

\begin{itemize}
\item  for all polynomial $p$ of degree $k-3$, and for every constant vector $\cc$, we have that 
$\bcurl(\cc b_4 p)$ belongs to $\Z_k$;
\item $\forall i$ with $1\le i\le 4$ and for all polynomial $p\in \P_{k-2}$ 
we have that 
$\bcurl(\nabla\lambda_i b_3^{(-i)}p)$  belongs to  $\Z_k$.
\end{itemize}

We can then introduce some additional notation.
For $s\ge 1$ integer,we set
\begin{itemize}
\item 
$\calB_s$ (the {\it bubbles} of degree $\le s$)$:=\{p\in \P_s(\E): p\equiv 0 \mbox{ on }\partial \E\}$.\\
Note that $\calB_s\equiv\{{ 0}\}$ for $s< 4$, and for  $s\ge 4$ the dimension
of $\calB_s$ is equal to  $\pi_{s-4,3}$. 

\end{itemize}
For $1\le i\le 4$ we define
 \begin{itemize}
\item 
$\calC_s^i$   (the {\it cups} of degree $\le s$)$:= \{ b_3^{(-i)} p_{s-3}\nn_i , \;p_{s-3}\in\P_{s-3}\}$.\\
Note that an element of $\calC^i_s$ vanishes on the whole $\partial\E$ with the only possible exception  of the face $f_i$. Moreover, $\calC^i_s\equiv\{{\bf 0}\}$ for $s< 3$, and for  $s\ge 3$ the dimension of each $\calC_s^i$ is equal to  $\pi_{s-3,3}$.

\end{itemize}
Then we set  
\begin{itemize}
\item $\calC_s:=span\{\calC_s^i (i=1,...,4)\}.$ 
\end{itemize}
We have the following result.
\begin{prop}\label{ZandC}
Let $E$ be a tetrahedron. Then it holds
$$
\Z_k = \bcurl(\calC_{k+1}) \: .
$$
\end{prop}
\begin{proof}
Since we already noted that $\bcurl(\calC_{k+1}) \subseteq \Z_k$, we only need to prove the converse. 
From \eqref{defZ3} it is easy to check that the space $\Z_k$ can be written as
\begin{equation}\label{Zk-alt}
\Z_k = \left\{ \bcurl \pp \: | \: \pp \in (\P_{k+1})^3 \mbox{ such that } \rot_f (\pp^{\tau_f}) = 0 \ \mbox{ for all } f \in \partial E
\right\} , 
\end{equation}
where we used the well known formula $ \rot_f \pp^{\tau_f}  =( \bcurl \pp_{|f}) \cdot {\bf n}_f$, valid on every face $f$ of $E$.

Therefore, in order to show that $\Z_k \subseteq \bcurl(\calC_{k+1})$, it is sufficient to prove that for any $\pp$ as in definition \eqref{Zk-alt} there exists an element $\bc_{k+1} \in \calC_{k+1}$ such that
$
\bcurl \bc_{k+1} = \bcurl\pp .
$
The above condition is surely satisfied if, given any $\pp$ as in definition \eqref{Zk-alt}, we can find $\bc_{k+1} \in \calC_{k+1}$ and $\psi \in \P_{k+2}(E)$ such that
\begin{equation}\label{L:target}
\bc_{k+1} + \bgrad \psi = \pp \quad \textrm{ in } E.
\end{equation}
We start by working on the boundary of the element. Take any face $f \in \partial E$. Since by definition $ \rot_f \pp^{\tau_f} = 0$, there exists a $q_f \in \P_{k+2}(f)$ such that
\begin{equation}\label{q-face}
\pp^{\tau_f} = \bgrad_f q_f \textrm{ on } f .
\end{equation}
Note that each $q_f$, $f \in \partial E$, is uniquely defined up to an additive constant; 
we also observe that across each edge $e$ of $\E$ (with faces $f,f'$ sharing $e$) it holds
\begin{equation}\label{eq-edge-tang}
\bgrad q_f\cdot {\bf t}_e = \pp_{|f}\cdot {\bf t}_e = \pp_{|{f'}} \cdot {\bf t}_e
= \bgrad q_{f'}\cdot {\bf t}_e .
\end{equation}
Running along the edges of each face (and taking into account that $\rot_f q_f=0$ on each face $f$), it is then easy to check that we can choose  the free additive constants  for  $q_f$ in such a way that they glue continuously across all edges. 

We can therefore define $\psi$ on each face as follows
\begin{equation}\label{face-def}
\psi_{|f} = q_f \quad \forall f \in \partial E ,
\end{equation}
and have that $\psi$ is continuous on $\partial\E$ and face-wise polynomial of degree $k+2$.

 Therefore, we can now take $\psi \in \P_{k+2}(E)$ in \eqref{L:target} as any polynomial having such a function as a trace (here is where we use the fact that $\E$ is a tetrahedron).  Note that if $k<2$ there is only one such polynomial, otherwise there are infinitely many, actually a space of dimension $\pi_{k-2,3}$.
Using \eqref{face-def}, \eqref{q-face} and recalling that the tangential components of elements in $\calC_{k+1}$ on the boundary are always vanishing, we have found that
\begin{equation}\label{eq-tang}
\forall f \in \partial\E, \
\forall \bar \bc_{k+1} \in \calC_{k+1}\qquad\Big( \bar \bc_{k+1} + \bgrad \psi \Big)^{\tau_f} = \Big( \bgrad \psi \Big)^{\tau_f} = \pp^{\tau_f},
\end{equation}
where the notation \eqref{vtauf} was used for the tangential components. Let now $e \in \partial E$ be an edge shared by two faces $f,f'$, and let ${\bf s}$ denote the unit vector co-planar with $f$ and orthogonal to $e$ (pointing outwards with respect to $f$). Similarly, let ${\bf s}'$ denote the analogous vector with respect to $f'$. It is immediate to check that $\{ {\bf t}_e, {\bf s}, {\bf s}' \}$ are linearly independent. Using \eqref{eq-edge-tang} and \eqref{eq-tang} we obtain that on the edge $e$
$$
\bgrad \psi \cdot{\bf t}_e = \pp\cdot{\bf t}_e , \quad
\bgrad\psi\cdot{\bf s} = \pp\cdot{\bf s} , \quad
\bgrad \psi\cdot{\bf s}' = \pp\cdot{\bf s}' .
$$ 
We conclude that in particular
\begin{equation}\label{edge-same}
 \bgrad \psi_{|e }=  \pp_{|e} \qquad
\forall e \in \partial\E .
\end{equation}
From \eqref{edge-same} we have that, for all $f_i$ in $\partial E$, the difference $(\bgrad \psi - \pp)_{|f_i} \cdot {\bf n} $ vanishes on the boundary of the face $f_i$, and is in $\P_{k+1}(f)$. Therefore such a function is a polynomial bubble of degree $k+1$ on the face; thus one can always find  a cup $\bc_{k+1}^i$ such that
\begin{equation*}
\big( \pp-\bgrad \psi \big)_{|f_i} \cdot {\bf n}^i = b_3^{(-i)} p_{k-2}^i =: \bc_{k+1}^i \cdot {\bf n}^i, \quad i=1,..,4 .
\end{equation*}
Note that each cup $\bc_{k+1}^i$ vanishes on all faces $f_j$ with $j \not= i$.
By taking   the function $\bar{\bc}_{k+1}=\sum_{i=1}^4 \bc_{k+1}^i$ and using \eqref{eq-tang} and \eqref{edge-same} we have  then
\begin{equation}\label{ABC}
\pp - \bgrad \psi  - \bar{\bc}_{k+1} = 0 \quad \textrm{ on } \partial E.
\end{equation}
The function on the left hand side of \eqref{ABC} is a polynomial in $(\P_{k+1})^3$ that vanishes on $\partial E$, and is therefore in the space of bubbles $(\calB_{k+1})^3$. Since $(\calB_{k+1})^3 \subset \calC_{k+1}$ we can find a $\hat{\cc}_{k+1} \in \calC_{k+1}$ such that
$$
\pp - \bgrad \psi - \bar{\bc}_{k+1} = \hat{\bc}_{k+1} \quad \textrm{ on } E.
$$
The proof is therefore concluded taking $\psi$ as above and $\bc_{k+1}=\bar{\bc}_{k+1}  + \hat{\bc}_{k+1}$.
\end{proof}

We now look into the dimension of the space $\bcurl (\calC_{k+1}) = \Z_{k}$.
We note that, for every $p\in\calB_{s}$ we have four cups $\nn_i p\in \calC_s^i~ (i=1,...,4)$, but only three of them are independent, as only three normals are independent.
Hence, in particular, it must hold
\begin{equation}
dim(\calC_s)~\le~4\,\pi_{s-3,3}-\pi_{s-4,3},
\end{equation}
that applied to $s=k+1$ gives
\begin{equation}
dim(\calC_{k+1})~\le~4\,\pi_{k-2,3}-\pi_{k-3,3}.
\end{equation}
According to what we saw in Proposition \ref{ZandC}, every
$\bcurl$ of an element of $\calC_{k+1}$ is an element of $\Z_{k}$. However, we  note that 
for every $\psi_{k-2}\in\P_{k-2}$  the gradient of $\nabla(b_4\psi_{k-2})$ belongs to $\calC_{k+1}$, and $\bcurl\nabla(b_4\,\psi_{k-2})$ is zero. Hence 
\begin{equation}\label{pippo1}
dim(\Z_k)=dim\big(\bcurl(\calC_{k+1})\big)
\le dim(\calC_{k+1})-\pi_{k-2,3} \le 3\pi_{k-2,3}-\pi_{k-3,3}.
\end{equation}
On the other hand, we easily obtain a lower bound on the dimension of $\Z_k$ by taking the dimension of $(\P_{k})^3$ and subtracting the number of constraints in \eqref{defZ3}. This is only a lower bound since in principle some of those constraints could be linearly dependent. Noting that the integral of the divergence must be equal to zero for any function with vanishing normal component on the boundary, one obtains
\begin{equation}\label{otherbound}
dim(\Z_k) \ge 3 \pi_{k,3} - 4 \pi_{k,2} - \pi_{k-1,3} + 1 = 3\pi_{k-2,3}-\pi_{k-3,3} ,
\end{equation}
where the last identity is trivial to check.
Combining bounds \eqref{pippo1} and \eqref{otherbound} we obtain that
$$
dim(\Z_k) = 3\pi_{k-2,3}-\pi_{k-3,3} .
$$

Other ad-hoc arguments could be applied for specific geometries. For instance it is almost immediate to check that on the unit cube $]-1,1[^3$ we have $\Z_1=\Z_2=\{\bf 0\}$ and
setting $b_6:=(x^2-1)(y^2-1)(z^2-1)$
$$\Z_3=\bcurl\left(span\Big\{(\frac{b_6}{x^2-1},0,0), (0,\frac{b_6}{y^2-1},0), (0,0,\frac{b_6}{z^2-1})\Big\}\right).$$

\subsection{The lazy choice and the stingy choice} As we have seen,  it is far from easy to design general properties that allow, for each single polyhedron, a simple and systematic strategy to spot the elements of $\Z_k$, and use them to chose the $\S_k$-preserving degrees of freedom. 
The ``simple'' available choices are essentially the lazy choice, and the systematic strategy
of subsection \ref{system} (with various prices depending on the type of slicing that we choose).

 In particular here the lazy choice would correspond to treat every polyhedron as if it was a tetrahedron, by picking, in an almost arbitrary way, four different planes that contain one or more faces each, and then construct the cups and the bubbles relative to the tetrahedron $T$ made by the four chosen planes. Clearly the number of these cups  and bubbles will depend on the desired accuracy $k$. Out of them we can then construct the elements 
of $\Z_k(T)$. To construct a suitable set of $\S_k$-preserving degrees of freedom we will keep all the boundary degrees of freedom \eqref{f3d_dof1} and all the ``divergence'' degrees of freedom \eqref{f3d_dof2kd}, then rearrange the other ones, inserting suitable ones based on $\Z_k(T)$: typically, integrals, over $T$, against all the elements of $\Z_k(T)$. Clearly, on a polyhedron with many faces, the true space $\Z_k$ will be much smaller, and our lazy choice will force us to use many more degrees of freedom than needed. 

Moreover the lazy choice, unfortunately, will not be available  when $\E$ is a parallelepiped (with three pairs of parallel planes). This happens since we cannot find four faces with four normals all different from each other (as needed to build a tetrahedron). On the other hand, the systematic strategy described in Subsection \ref{system} is {\it always} a way-out, although it might require a heavy additional work
on each element (that in our opinion would be worth the effort only in very special cases, and in particular if one plans to use the same mesh for many different computations).

\begin{remark}\label{BconS-f3}
It is easy to see that, similarly to what has been done for face 2d elements (and extended to edge 2d elements), here too we can easily construct a B-compatible interpolation operator, that will work both for the original face 3d spaces and for their Serendipity variant. See again \eqref{Bcompatf2} and Remark \ref{BconS}.\qed
\end{remark}


\

\section{Edge elements in 3d}\label{edge3d}

The definition of edge elements in three dimensions is more complex than the above, and requires suitable VEM spaces on the faces, and suitable VEM spaces inside. 

\subsection{The boundary}

At a (very) general level, for every triplet ${\boldsymbol\beta}=(\beta,\beta_d,\beta_r)$ and for every face $f$ we set
\begin{equation}
\bV^e_{\boldsymbol \beta}(f):=\bV^e_{\beta,\beta_d,\beta_r}(f)
\end{equation}
and we define
\begin{multline}\label{Be3d}
{\bold B}_{{\boldsymbol\beta}}(\partial\E):=\!\{\vv|\vtauf\!\!\in\!\!\bV_{\boldsymbol \beta}^e(f)
\forall \mbox{ face }f\!
\mbox{ and }\! \vv\cdot\tt_e \!\mbox{ continuous $\!\forall$  edge $e$  of $\partial\E$}
\}.
\end{multline}
%
\subsection{The curl}
For every triplet ${\boldsymbol \mu}=(\mu,\mu_d,\mu_r)$ we set
\begin{equation}
\bV^f_{\boldsymbol \mu}(\E):=\bV^f_{\mu,\mu_d,\mu_r}(\E).
\end{equation}
\subsection{The space}
We are ready: for indexes ${\boldsymbol\beta},\, k_d,\,{\boldsymbol\mu}$ with
$\beta_r=\mu$ we define
\begin{equation}\label{VEMe3d}
\bV^e_{{\boldsymbol\beta},k_d,{\boldsymbol\mu}}(\E):=\{\vv|\,\mbox{s. \!t. } \vv_{|\partial\E}\in {\bold B}_{{\boldsymbol\beta}}(\partial\E); \div\vv\in\P_{k_d}(\E),\,\bcurl\,\vv\in\bV^f_{\boldsymbol\mu}\}.
\end{equation}

Note that the equality $\beta_r=\mu$ {\bf must} be required because, on every face $f$, we have that $\rot_f\vtauf$ (that belongs to $\P_{\beta_r}(f)$) coincides with $(\bcurl\,\vv)\cdot\nn_f$ (that belongs
to $\P_{\mu}(f)$), that is
\begin{equation}\label{Stokes}
\rot_f\vtauf\equiv(\bcurl\,\vv)\cdot\nn_f\equiv\ww\cdot\nn_f.
\end{equation}
This  can be easily seen by considering a face $f$ with equation $x_3=0$ where $\bcurl\,\vv\cdot\nn_f$ (the third component of $\curl\,\vv$) is given by $v_{2,x}-v_{1,y}\equiv\rot_f\vtauf$.
Moreover, since the divergence of any $\bcurl \,\vv \in \bV^f_{\boldsymbol \mu}(\E)$ vanishes, one can directly take $\mu_d=-1$ in the definition of ${\boldsymbol\mu}$. As a consequence of the above observations, we always have $\mu=\beta_r$ and $\mu_d=-1$. 
Therefore the space $\bV^e_{{\boldsymbol\beta},k_d,{\boldsymbol\mu}}(\E)$ in \eqref{VEMe3d} is determined by five (and not seven) parameters.

As far as the degrees of freedom are concerned, we need, at the boundary:
\begin{align}
\label{e3d_dofb1}
&\bullet \int_e\vv\cdot\tt_e\,q_{\beta}\de\quad\mbox{ for all $q_{\beta}\in\P_{\beta}(e)$, for all edge $e$},\\
&\bullet \mbox{ for  $\beta_d\ge 0$:}\quad 
\label{e3d_dofb2}
 \int_{f}\vv\cdot\xx \,q_{\beta_d}\df \quad\mbox{ for all $q_{\beta_d}\in\P_{\beta_d}(f)$ for all face $f$},\\
&\bullet \mbox{ for  $\beta_r\ge 1$: }\quad
\label{e3d_dofb3}
\int_{f}\vv\cdot\brot q_{\beta_r}\df \quad\mbox{ for all $q_{\beta_r}\in\P_{\beta_r}(f)$ for all face $f$.}
\end{align} 
As we observed in the two-dimensional case (see \eqref{prokr2d}) we see that out of the above degrees of freedom we will be able to compute,  for each $\vv\in
{\bold B}_{{\boldsymbol\beta}}(\partial\E)$:
\begin{equation}\label{e3d_dofb3}
\int_{f}\vtauf\cdot \qq_{s}\df \quad\forall\mbox{  face $f$ and 
$\forall~\qq_{s}\in(\P_{s}(f))^2$, for $s\le\beta_d+1$}.
\end{equation}
As far as $\ww:=\bcurl\,\vv$ is concerned, we should use \eqref{f3d_dof1}-\eqref{f3d_dof2kr}. We note however that, always for $\mu=\beta_r$ and using \eqref{Stokes}, the d.o.f.  \eqref{f3d_dof1} are already determined by the values of $\rot_f\vtauf$ on each face, that in turn can
be computed using \eqref{e3d_dofb3} and \eqref{e3d_dofb1}. Similarly, the d.o.f \eqref{f3d_dof2kd} (after integration by parts) are equal to $\int_{\partial\E}\ww\cdot\nn \,q_{\mu_d}$, since obviously $\div\ww=0$. Hence, the only information that is needed, in addition 
to \eqref{e3d_dofb1}-\eqref{e3d_dofb3} is: 
\begin{equation}
 \bullet \mbox{ for $\mu_r\ge 0$}:\quad
\label{e3d_dofcu3}
\int_{\E}\ww\cdot\xx\wedge\,\qq_{\mu_r}\dE \quad\mbox{ for all $\qq_{\mu_r}\in (\P_{\mu_r}(\E))^3$}.
\end{equation}
Following the previous discussion (see formula \eqref{prokr3d}) we see that out of the above degrees of freedom we will be able to compute,  for each $\vv\in
 \bV^e_{{\boldsymbol\beta},k_d,{\boldsymbol\mu}}(\E)$:
\begin{equation}\label{procurl}
\int_{\E}(\bcurl\vv)\cdot \qq_{s}\dE \quad\forall~\qq_{s}\in(\P_{s}(\E))^3, \mbox{for $s\le\mu_r+1$}.
\end{equation}

After we took care of $\ww\equiv\bcurl\,\vv$ we must (finally) require
\begin{equation}\label{e3d_dofd1}
\bullet \mbox{ for  $k_d\ge 0$:}\quad\int_{\E}\vv\cdot\xx\,q_{k_d}\dE \quad\mbox{ for all $q_{k_d}\in\P_{k_d}(\E)$}.
\end{equation}
\begin{remark}
If we need to compute the projection of an element $\vv\in \bV^e_{{\boldsymbol\beta},k_d,{\boldsymbol\mu}}(\E)$ onto the space $(\P_s(\E))^3$ 
we can use the decomposition \eqref{decompPk3cx} as $$\pp_s=\bcurl\qq_{s+1}
+\xx r_{s-1}$$ and, integrating the first term by parts, write
\begin{multline}\label{proe3d}
\int_{\E}\vv\cdot\pp_s=\int_{\E}\vv\cdot\bcurl\qq_{s+1}+\int_{\E}\vv\cdot\xx r_{s-1}\\
=\int_{\E}\bcurl\vv\cdot\qq_{s+1}+\int_{\partial\E}\vtauf\cdot\qq_{s+1}\wedge\nn
+\int_{\E}\vv\cdot\xx r_{s-1}.
\end{multline}
In the last line, the first term, as in \eqref{procurl}, can be computed for $s+1\le\mu_r+1$
(meaning obviously $s\le\mu_r$); the second term, as in \eqref{e3d_dofb3}, can be computed for $s+1\le \beta_d+1$ (meaning, here too, $s\le\beta_d$) and finally the
last term, following \eqref{e3d_dofd1}, can be computed for $s-1\le k_d$, meaning
$s\le k_d+1$. Summarizing: the projection of an element $\vv\in\bV^e_{{\boldsymbol\beta},k_d,{\boldsymbol\mu}}(\E)$ onto the space $(\P_s(\E))^3$ can be computed for
$$ s\le \min\{\mu_r,\beta_d,k_d+1\}.$$
\end{remark}

\begin{remark} \label{altBpres} In a case like the present one (in which the space 
$\bcurl(\bV_{{\boldsymbol\beta},k_d,{\boldsymbol\mu}}^e(\E))$ 
is {\bf not} a polynomial space),
the {``B-compatibility property''} (see e.g.  \eqref{Bcompatf2}) would be better defined, for an interpolation
operator $\Pi$ from $(C^1(\E))^3 $ to $\bV_{{\boldsymbol\beta},k_d,{\boldsymbol\mu}}^e(\E)$, as
\begin{equation}\label{Bcompate3}
\forall \uu\in (C^1(\E))^3 \mbox{ with }\bcurl\,\uu\in \bV^f_{\boldsymbol\mu}(E)\mbox{ we have } \bcurl(\Pi\uu-\uu)=0.
\end{equation}
With that, we easily see that the natural interpolation operator associated with the degrees of freedom \eqref{e3d_dofb1}-
\eqref{e3d_dofb3}, \eqref{e3d_dofcu3}, and \eqref{e3d_dofd1} is $\bcurl$-preserving.\qed
\end{remark}







All this is, dealing with spaces with seven indexes, is {\bf very} general, and {\bf very} confusing. We shall therefore look at some particular case. 

\subsection{A particular case: {\it N2}-type VEMs}
We set ${\boldsymbol \beta}=(k,k-1,k-1)$,  ${\boldsymbol \mu}=(k-1,-1,k-2)$, and $k_d=k-1$. Then we have for each face, the {\it N2}-like VEM space:
\begin{multline}\label{8.12}
\bV^e_{\boldsymbol \beta}(f):=\bV^e_{k,k-1,k-1}(f)\equiv\{\vv|\,\mbox{such that }\\ \vv\cdot\tt_e\in\P_k(e) \,\forall \mbox{ edge } e,\; \div_f\vv\in\P_{k-1},\,\rot_f\vv\in\P_{k-1}\}.
\end{multline}
Note that, for a triangular face, we will have the space $(\P_k(f))^2$.
{ The space ${\mathcal B}_{{\boldsymbol\beta}}(\partial\E)$ will then be} made of vector valued functions that on each edge have a tangential component of degree $\le k$, and whose tangential part on each face has a
 2d divergence and a  2d rotational polynomials of degree $\le k-1$. Moreover the tangential components on edges
are continuous (= single valued) when passing from a face to a neighboring one.

As degrees of freedom in ${\mathcal B}_{{\boldsymbol\beta}}(\partial\E)$ we have
\begin{align}
\label{e30_dof1}
&\bullet\int_e\vv\cdot\tt_e\,q_k\ds
\mbox{ on each edge $e$, for each $q_k\in\P_k(e)$},\\
\label{e30_dof2d}
&\bullet\int_f\vtauf\cdot\xx^{\tau_f}\,q_{k-1}\df \mbox{ on each face $f$, for each $q_{k-1}\in \P_{k-1}(f)$},\\
\label{e30_dof2r}
&\bullet\int_f\vtauf\cdot\brot_2 q_{k-1}\df \mbox{ on each face $f$, for each $q_{k-1}\in \P_{k-1}(f)$}.
\end{align}
 As additional degrees of freedom  for $\ww\equiv\bcurl\,\vv$ in 
$\bV^f_{\boldsymbol \mu}(\E)\equiv \bV^f_{k-1,-1,k-2}(\E)$ we have, according to \eqref{e3d_dofcu3}
\begin{equation}\label{e3d_dofcu3-2}
\hskip-1.8truecm\bullet\int_{\E}\ww\cdot\xx\wedge\,\qq_{k-1}\dE \quad\mbox{ for all $\qq_{k-1}\in (\P_{k-1}(\E))^3$}.
\end{equation}
Finally we will need the degrees of freedom \eqref{e3d_dofd1} that now become
\begin{equation}\label{e3d_dofdk4}
\hskip-2.9truecm\bullet\int_{\E}\vv\cdot\xx\,q_{k-1}\dE \quad\mbox{ for all $q_{k-1}\in\P_{k-1}(\E)$}.
\end{equation}
One can see that this could be interpreted as  a  generalization to polyhedrons of
the N\'ed\'elec second-kind elements.

We point out that
the space defined in \eqref{8.12} is exactly the same three-dimensional edge space introduced in \cite{super-misti}, while the degrees of freedom are different. 

\subsection{Another particular case: {\it N1}-type spaces}
The Virtual Elements of the previous subsection were of the $BDM$ or $N2$ type. Let us see here those of $RT$ or $N1$ type.

We set ${\boldsymbol \beta}=(k,k-1,k)$,  ${\boldsymbol \mu}=(k,-1,k-1)$, and $k_d=k-1$. Then we have for each face:
\begin{multline}
\bV^e_{\boldsymbol \beta}(f):=\bV^e_{k,k-1,k}(f)\equiv\{\vv|\,\mbox{ such that }\\ \vv\cdot\tt_e\in\P_k(e) \,\forall \mbox{ edge } e,\; \div_f\vv\in\P_{k-1},\,\rot_f\vv\in\P_{k}\}.
\end{multline}
 The space ${\mathcal B}_{{\boldsymbol\beta}}(\partial\E)$ will be made of vector valued functions that on each edge have a tangential component of degree $\le k$, and whose tangential part on each face has a
 2d divergence of degree $k-1$ and a  2d rotational of degree $\le k$. Moreover the tangential components on edges
are continuous (= single valued) when passing from one face to a neighboring one.

As degrees of freedom in ${\mathcal B}_{{\boldsymbol\beta}}(\partial\E)$ we have
\begin{align}
\label{e30_dof1RT}
&\bullet\int_e\vv\cdot\tt_e\,q_k\ds
\mbox{ on each edge $e$, for each $q_k\in\P_k(e)$},\\
\label{e30_dof2dRT}
&\bullet\int_f\vtauf\cdot\xx^{\tau_f}\,q_{k-1}\df \mbox{ on each face $f$, for each $q_{k-1}\in \P_{k-1}(f)$},\\
\label{e30_dof2rRT}
&\bullet\int_f\vtauf\cdot\brot_2 q_{k}\df \mbox{ on each face $f$, for each $q_{k}\in \P_{k}(f)$}.
\end{align}
 As additional degrees of freedom  for $\ww\equiv\bcurl\,\vv$ in 
$\bV^f_{\boldsymbol \mu}(\E)\equiv \bV^f_{k-1,-1,k-1}(\E)$ we have, according to \eqref{e3d_dofcu3}
\begin{equation}\label{e3d_dofcu3RT}
\hskip-2.8truecm\bullet\int_{\E}\ww\cdot\xx\wedge\,\qq_{k}\dE \quad\mbox{ for all $\qq_{k}\in (\P_{k}(\E))^3$}.
\end{equation}
Finally we will need the degrees of freedom \eqref{e3d_dofd1} that now become
\begin{equation}\label{e3d_dofdk4RT}
\hskip-2.9truecm\bullet\int_{\E}\vv\cdot\xx\,q_{k-1}\dE \quad\mbox{ for all $q_{k-1}\in\P_{k-1}(\E)$}.
\end{equation}
One can see that this
could be interpreted as  a  generalization to polyhedrons of
the {\it N1} elements.

\subsection{Unisolvence of the degrees of freedom}\label{unise3d}
For the sake of simplicity, we will discuss the unisolvence of our degrees of freedom for 3d edge Virtual Elements of the type {\it N2}. The extension to the general case would be conceptually trivial and only the notation would be heavier.

Assume therefore that, for a particular $\vv$ in our space,  all the degrees of freedom \eqref{e30_dof1}-\eqref{e3d_dofdk4} are zero. Using the degrees of freedom \eqref{e30_dof1}-\eqref{e30_dof2r}  (which are on each face the analogues of \eqref{e2d_d0f4}-\eqref{e2d_d0f6}) we easily see that 
on each face $f$ the tangential component $\vtauf$ is identically zero, thanks to Proposition \ref{unidofe2d}.
Hence, the normal component
of $\ww=\bcurl\,\vv$ will also be zero on each face, that is, the d.o.f. \eqref{f3d_dof1} for $\ww$ are zero.  We also have, integrating by parts and using $\div\ww=0$,  
\begin{equation}\label{uni3d2}
{\color{white}.}\hskip-0.5truecm  \int_{\E}\ww\cdot\bgrad q_{k-1}\dE =\!\sum_{f}\int_{f}\ww\cdot\nn\, q_{k-1}\df = 0\mbox{ for all } q_{k-1}\!\in\!\P_{k-1}(\E),
\end{equation}
so that the d.o.f. \eqref{f3d_dof2kd} for $\ww$ are also zero. Finally, since the d.o.f.
\eqref{e3d_dofcu3-2} are equal to zero, we have that \eqref{f3d_dof2kr} for $\ww$ are also zero. The unisolvence of the degrees of freedom \eqref{f3d_dof1}-\eqref{f3d_dof2kr} for face elements implies then $\ww\equiv\bcurl\,\vv=0$. Therefore,  $\vv=\bgrad\vphi$ for some $\vphi\in H^1(\E)$,
and since $\vtauf=0$ on the boundary we can take  $\vphi\in H^1_0(\E)$. 
As $\div\vv\in\P_{k-1}$ we have that $\Delta\vphi\in\P_{k-1}$. Since
we assumed that the degrees of freedom \eqref{e3d_dofdk4} are also zero, and recalling \eqref{divxq3} we deduce that 
\begin{equation}
\int_{\E}\vphi\, p_{k-1}\dE=0\quad\forall p_{k-1}\in \P_{k-1} .
\end{equation}
Thus,
\begin{equation}
\int_{\E} |\nabla \vphi|^2 \dE=-\int_{\E}\vphi\,\Delta \vphi  \dE  =0,
\end{equation}
so that
$\vphi=0$ and hence $\vv=0$.

With minor modifications, the above proof can be adapted to the general case given in \eqref{VEMe3d}.



\subsection{Serendipity Edge Virtual Elements in 3d}\label{sere-face-2D}

Following the same path of the previous sections, we could now construct the serendipity variants of
our 3d edge VEMs.

 We remark however, from the very beginning, that (contrary to what happened for {\it face} 3d elements), here on each face we have a  2d VEM space (and not just a polynomial as we had in in \eqref{VEMf3d}). We also point out that there is a big difference, for three-dimensional elements,  between the degrees of freedom {\it internal to the element} (that could be eliminated by static condensation) and the degrees of freedom
 {\it internal to faces}, where static condensation cannot be applied).  
 
 We also point out  that in general the number of faces is quite big: for instance, on a regular mesh of 
 $n\times n\times n$ cubes we have $n^3$ cubes and, asymptotically, $3n^3$ faces (precisely $3n^3+3n^2$, including the boundary ones).

Hence it would  be very convenient, whenever possible, to use, on faces,  the 2d serendipity spaces (introduced in Subsection {\ref{sere-edge-2D}})  instead of the original ones from Subsection \ref{edge2d-1}. In order to describe the Serendipity reduction for the present three-dimensional edge elements, we could choose for simplicity one of the two classical cases ($N1$-like VEMs or $N2$-like  VEMs), or else remain in the more general context of the space \eqref{VEMe3d}. Following what we did in Subsection
\ref{VEMe3d} we take the simplest case of $N2$-like VEMs, in the hope that once this case is clear the more complex ones could be reconstructed without major efforts.

Hence, we start by changing \eqref{Be3d} into 
\begin{multline}\label{Be3dS}
{\mathcal B}_{{k}}^S(\partial\E):=\{\vv| \mbox{ such that }\vtauf\in \bV_{\S_k}^e(f)\,
\mbox{ for all face } f \mbox{ of $\partial\E$ }\\
\mbox{ and } \vv\cdot\tt_e \mbox{ continuous along the edges $e$ of
$\partial\E$}\},
\end{multline}
where with our choice the space $\S_k$ to be preserved, on each face $f$, is  $N2(f)$, and  $\bV_{\S_k}^e(f)$ is the corresponding 2d serendipity edge space constructed in Sect. \ref{edge2d} (that is, precisely, \eqref{defVSe2d} with $\beta_r$ equal to  $k-1$). 

 If, for the sake of simplicity, every face $f$ of $\E$ is a convex polygon, we can now apply the general strategy of Sect. \ref{seregen} to each face. Thus we have:  
 \begin{prop} Assume that every face $f$ of $\E$ is a convex polygon, and let $\eta_{f}$ be defined as in \eqref{defeta}. Then in ${\mathcal B}_{{k}}^S(\partial\E)$ we can use the degrees of 
freedom 
\begin{align}
\label{f2d_SERE1ef}
&\bullet\int_e\vv\cdot\tt_e\,q_k\de,\quad\mbox{ $ \forall\,q_{k}\in\P_k(e)$}, \; \forall \mbox{ edge $e$},\\
\label{f2d_SERE2ef}
&\bullet\mbox{ for }k\ge 2: \int_{f}\vv\cdot\brot_{f} \,q_{k-1}\df, \quad \forall q_{k-1}\in\P_{k-1}(f),\,\forall\mbox{  face $f$},
\end{align}
plus, whenever $s:=k+1-\eta_f$ is non-negative,
\begin{equation}\label{f2d_SERE3ef}
\hskip-1.4truecm\bullet\;\int_{f}\vv\cdot\xx\,q_{s}\df, \quad\forall q_{s}\in\P_{s}(f) \,\forall\mbox{  face $f$} .
\end{equation}
\end{prop}
Then, as a starting space, we use
\begin{equation}\label{VEMe3dS}
\Y^e_{k}(\E)\:=\!\{\vv|\,\mbox{s.t. } \vv_{|\partial\E}\!\in\! {\mathcal B}_{{k}}^S(\partial\E); \div\vv\!\in\!\P_{k-1}(\E),\,\bcurl\,\vv\!\in\!\bV^f_{k,-1,k-1}\}.
\end{equation}

Starting from $\Y^e_{k}(\E)$, and following our choice of $N2$-like VEMs,  we  now choose the polynomial space (that we still denote by $\S_k$) that we want to preserve,
as $ N2_k(\E) $ (that is, $(\P_k(\E))^3$).  Then, following the track of the previous cases, we   start our ``Serendipity reduction''  by choosing a suitable set of  degrees of freedom that we wnt to keep. In particular,  (as in Subsection \ref{keepdof}), we will choose the boundary ones
\eqref{f2d_SERE1ef}-\eqref{f2d_SERE3ef} (that are the boundary-serendipity substitutes of the \eqref{e3d_dofb1}, \eqref{e3d_dofb3}) to provide
$H(\bcurl)$-conformity,  together with the ones in \eqref{e3d_dofcu3RT} to ensure B-compatibility (where, this time,  $B$ is the $\bcurl$ operator). In case these are not
$\S_k$-{ identifying}, we will have to choose some additional ones among the 
\eqref{e3d_dofd1}.

\subsection{{\bf Boundary preserving}, $\bcurl$-preserving, and $\S_k$-{\bf identifying} degrees of freedom}
In order to decide which degrees of freedom to choose, we must start considering the vector-valued polynomials $\pp$, of degree $\le k$, that have the
degrees of freedom \eqref{f2d_SERE1ef}-\eqref{f2d_SERE3ef} and \eqref{e3d_dofcu3RT}
equal to zero.  We define therefore
\begin{equation}
\Z_k:=\{\pp\in (\P_k(\E))^3\mbox{ s.t. }\pp^{\tau_f}=0
\mbox{ on }\partial\E\mbox{ and }\bcurl\,\pp=0\}.
\end{equation}
It is almost immediate to see that all the elements $\pp$ of $\Z_k$ must be gradients (since their $\bcurl$ is equal to zero and $\E$ is simply connected). Hence $\pp=\bgrad\vphi$ for some $\vphi\in\P_{k+1}$. Considering the boundary conditions
we see that we can take $\vphi\in H^1_0(\E)$. In other words, $\vphi$ will be a {\it scalar bubble of degree $k+1$}. Recalling the results of \cite{SERE-Nod} we can define now
$\eta_{\E}$ as {\it the minimum number of different planes necessary to cover
$\partial\E$}, and deduce that if $\eta_{\E}>k+1$ then $\Z_k$ is reduced to $\{{\bf 0}\}$, and the degrees of freedom 
\eqref{f2d_SERE1ef}-\eqref{f2d_SERE3ef} and \eqref{e3d_dofcu3RT}
 will already be able to identify
all the elements of $\S_k$ in a unique way; this would mean that we can take $\SN=\MM$
in \eqref{Sident}. Otherwise, for $\eta_{\E}\le k+1$, we will have that the dimension of $\Z_k$ is equal to $\pi_{k+1-\eta,3}$, and we need an $\SN$ such that
$\SN-\MM\ge \pi_{k+1-\eta,3}$. As in the two-dimensional cases (and also for nodal Serendipity
VEMs) we have now that, for a convex $\E$ we could take as additional degrees of freedom

\begin{equation}\label{e3d_dofdk4S}
\int_{\E}\vv\cdot\xx\,q\dE \quad\mbox{ for all $q\in\P_{k+1-\eta}(\E)$}. 
\end{equation}
 Othewise, in the non-convex case, the easiest way out would probably be to follow the systematic path  of Subsection
 \ref{system} and start by checking whether the $\calD$ matrix corresponding to the above choice \eqref{e3d_dofdk4S} has maximum rank or not. If you are not particularly unlucky, it will, and you can behave
 as in the convex case. Otherwise, you could add (say, one by one) the degrees of freedom
 of type \eqref{e3d_dofdk4S} corresponding to a $q$ homogeneous polynomial of degree
 $k+2-\eta$ (and if all of them fail, you pass to the homogeneous degree $k+3-\eta$, and so
 on). Or else, you pick a {\it lazy choice} and use directly \eqref{e3d_dofdk4S} with
 all the $q$ in $\P_{k-3}$ (as if $\E$ was a tetrahedron). 
 
%

Note that, in several cases, the {\bf gain} in the number of degrees of freedom
(compared to the general case \eqref{VEMe3d})  will not be due to the reduction
of the degrees of freedom in \eqref{e3d_dofdk4S} using polynomials of degree  $k+1-\eta_{\E}$ instead of the original  $k-1$, but mostly to the choice of using Serendipity edge VEMs on {\it faces}.



\bibliographystyle{amsplain}

\bibliography{general-bibliography}

\end{document}